\documentclass{amsart}
\usepackage[utf8]{inputenc}
\usepackage{amsmath,amssymb,amsthm}
\usepackage{amsaddr}
\usepackage{enumitem}
\usepackage{xcolor}
\usepackage{tabularx}
\usepackage{url}
\usepackage[unicode=true]{hyperref}
\usepackage{bm}
\usepackage{tikz-cd}

\DeclareMathOperator{\Aut}{Aut}
\DeclareMathOperator{\End}{End}

\DeclareMathOperator{\fsl}{\mathfrak{sl}}

\DeclareMathOperator{\rk}{rk}
\DeclareMathOperator{\Mat}{Mat}
\DeclareMathOperator{\Tor}{Tor}

\DeclareMathOperator{\Hom}{Hom}
\DeclareMathOperator{\Sym}{Sym}

\DeclareMathOperator{\Image}{im}
\DeclareMathOperator{\Ker}{ker}

\DeclareMathOperator{\HH}{H}

\newcommand{\bbf}{\mathbb{F}}
\newcommand{\bbz}{\mathbb{Z}}
\newcommand{\bbn}{\mathbb{N}}
\newcommand{\bbq}{\mathbb{Q}}
\newcommand{\bbc}{\mathbb{C}}
\newcommand{\bbr}{\mathbb{R}}

\newcommand{\ffg}{\mathfrak{g}}

\newcommand{\ffh}{\mathfrak{h}}

\newcommand{\Zalgebra}[1]{\bbz_p[\![#1]\!]}
\newcommand{\Qalgebra}[1]{\bbq_p[\![#1]\!]}
\newcommand{\scheme}[1]{\mathsf{#1}}
\newcommand{\schG}{\scheme G}

\newcommand{\cotimes}{\widehat{\otimes}}

\newcommand{\Wmodk}[1]{\scheme{W}_{\bm\lambda}^{#1}}
\DeclareMathOperator{\GL}{\scheme{GL}}
\DeclareMathOperator{\SL}{\scheme{SL}}

\theoremstyle{plain}
\newtheorem{thm}{Theorem}[section]
\newtheorem{lem}[thm]{Lemma}
\newtheorem{cor}[thm]{Corollary}
\newtheorem{prop}[thm]{Proposition}
\newtheorem{conj}[thm]{Conjecture}

\theoremstyle{definition}
\newtheorem{rmk}[thm]{Remark}
\newtheorem{defn}[thm]{Definition}

\title{Asymptotics of rational representations for algebraic groups}
\author{Lander Guerrero Sánchez}
\address{Department of Mathematics, Euskal Herriko Unibertsitatea, Bilbao, Spain}
\author{Henrique Souza}
\address{Department of Mathematics, Universidad Autónoma de Madrid, Madrid, Spain}
\email{henrique.mendesdasilva@uam.es}
\date{\today}

\begin{document}

\begin{abstract}
    We study the asymptotic behaviour of the cohomology of subgroups \(\Gamma\) of an algebraic group \(\schG\) with coefficients in the various irreducible rational representations of \(\schG\) and raise a conjecture about it. Namely, we expect that the dimensions of these cohomology groups approximate the \(\ell^2\)-analogues of Betti numbers of \(\Gamma\) with a controlled error term. We provide positive answers when \(\schG\) is a product of copies of \(\SL_2\).
\end{abstract}

\maketitle

\section{Introduction}

Discrete subgroups of semi-simple Lie groups and their cohomology are ubiquitous in modern mathematics. In the context of geometry and topology, they appear as fundamental groups of locally symmetric Riemannian manifolds and their cohomology groups have a natural interpretation in terms of differential forms. In the context of number theory, arithmetic lattices encode critical information about number fields and their cohomology is related to the theory of automorphic forms. At this point, it goes without saying that there is an extensive and deep literature on this broad topic, and just to name a few major works we cite the classical books of M. S. Raghunathan \cite{raghunathanDiscreteSubgroupsLie1972} and A. Borel and N. Wallach \cite{borelContinuousCohomologyDiscrete2000}.

Within those contexts, there is significance in the study of the asymptotic behaviour of the cohomology groups \(\HH^i(\Gamma,\scheme{W})\) when \(\Gamma\) is a discrete subgroup of a semi-simple Lie group \(\schG\) and \(\scheme{W}\) is a finite dimensional complex representation of \(\schG\). This asymptotic behaviour could be considered either with respect to going down a residual chain of finite index subgroups of \(\Gamma\) or with respect to taking larger and larger representations \(\scheme{W}\) of \(\schG\). The former is the realm of Lück approximation (\cite{luck_approximation_1994,jaikin-zapirain_base_2019,boschheidgenTwistedL2Betti}, cf. Section~\ref{sec-vonNeumann}), and we are concerned with the latter. 

Therefore, given a subgroup \(\Gamma\) and a sequence of finite dimensional complex representations \(\{\scheme{W}_k\}\) of a semi-simple Lie group \(\schG\) with \(\lim_{k \to \infty} \dim \scheme{W}_k = \infty\), we are interested in the following questions:

\begin{enumerate}
    \item Does the sequence \begin{equation}\label{eq-normalized-cohomology}
        \left\{\frac{\dim \HH^i(\Gamma,\scheme{W}_k)}{\dim \scheme{W}_k}\right\}\tag{\(*\)}
    \end{equation} converge and, if so, can we describe the limit?
    \item Does this limit depend on the choice of \(\scheme{W}_k\)?
    \item Can we estimate the rate of convergence?
\end{enumerate}

The starting point of this work is the paper of W. Fu \cite{fu}.  He proves that if
\(\Gamma\) is an arithmetic lattice in \(\schG=\prod_{i=1}^k \SL_2\), then the growth of the cohomology groups \(\HH^i(\Gamma,\scheme{W}_k)\) along natural families of algebraic coefficient systems is strictly smaller than the trivial bound \(O(\dim \scheme{W}_k)\); see \cite{fu} for the precise formulation.  For certain arithmetic lattices such as Bianchi groups, Fu's upper bound is asymptotically sharp. In the discussion following Fu's talk at Princeton on March~30,~2022, P.~Sarnak asked whether analogous growth bounds are known for non-arithmetic lattices in \(\SL_2\) \cite[41:58]{FuIASGrowthBianchi2022}.

Motivated by this question, our initial goal was to extend Fu's results beyond
the arithmetic setting.  We accomplish this here, thereby answering Sarnak's
question, and we expect the underlying method to extend to higher rank.  In
particular, we formulate conjectural answers (in the same spirit) for the
complex points of a semisimple algebraic group \(\schG\) and for subgroups
\(\Gamma\leq \schG\) of type \(FP_\infty\) (not assumed arithmetic, and not
necessarily discrete).  We also prove these conjectural bounds when \(\schG\) is
a product of copies of \(\SL_2(\bbc)\).

We expect that the answer to questions (1) and (2) depends on two factors. The first factor is how the central elements \(\scheme Z_\Gamma = \scheme Z(\schG)\cap \Gamma\) act on each representation. More precisely, each irreducible representation \(\scheme{W}\) of \(\schG\) induces a central character \(\chi\colon \scheme Z(\schG) \to \bbc^\times\). If~(\ref{eq-normalized-cohomology}) is to converge for a sequence of irreducible representations \(\{\scheme W_k\}\), we need them to induce the same central character when restricted to \(\Gamma\). An explicit illustration of this fact is given at the end of Remark~\ref{rmk-twisted-torsion-free}. Clearly this is not an obstruction if \(\Gamma\) is torsion-free, since then \(\scheme Z(\schG)\) intersects \(\Gamma\) trivially. Hence, we assume that the restriction \(\chi\colon \scheme Z_\Gamma \to \bbc^\times\) is constant along the family \(\scheme{W}_k\).

The second factor regards in which direction along the weight lattice of \(\schG\) is the sequence of representations \(\{\scheme W_k\}\) growing. We recall that to each irreducible representation \(\scheme W\) of \(\schG\) one can associate a tuple \(\bm\lambda = (\lambda_1,\ldots,\lambda_n)\) of non-negative integers called the highest weight of \(\scheme W\), and this highest weight characterizes \(\scheme W\) up to a \(\schG\)-equivariant isomorphism. To ensure convergence, we ask that the representations \(\scheme W_k\) are all irreducible and that all their weight parameters \(\lambda_i = \lambda_i(k)\) grow to infinity, i.e. \(\lim_{k\to \infty}\min \lambda_i = \infty\). For the rest of the introduction, we assume that this is the case.

The description of the limit is given in terms of S. Kionke's \(\phi\)-Betti numbers (\cite{Kionke2018}) for a character \(\phi \in \ell^\infty(\Gamma)\). Each irreducible character \(\chi\colon \scheme Z_\Gamma \to \bbc^\times\) induces a character \(\phi = \operatorname{ind}^\Gamma_{\scheme Z_\Gamma} \chi\) of the group \(\Gamma\), which in turn gives a map from \(\bbc[\Gamma]\) to a finite von Neumann algebra \(\mathcal{N}_\phi\) (see \cite[Def. 5.1]{Kionke2018} or Section~\ref{sec-vonNeumann}). This von Neumann algebra \(\mathcal{N}_\phi\) induces a Sylvester rank function on \(\bbc[\Gamma]\) (see Section~\ref{sec-liemodules}), and following Kionke we define the \(i\)-th \(\phi\)-Betti number of \(\Gamma\) as \[b_i^{\phi}(\Gamma) = \mathrm{null}_\phi(\partial_n) - \rk_\phi(\partial_{n+1})\] where \((Q_\bullet, \partial_n)\) is any finitely generated projective resolution of \(\bbc\) over \(\bbc[\Gamma]\). If \(Z_\Gamma\) is trivial, then \(b_i^{\phi}(\Gamma)\) coincides with the usual \(\ell^2\)-Betti numbers \(b_i^{(2)}(\Gamma)\) of \(\Gamma\).

For the statement of the conjectures, we fix an arbitrary semi-simple algebraic group \(\schG\) over \(\bbc\) and we assume that \(\Gamma\) is a subgroup of \(\schG\) of type \(FP_\infty\) (not necessarily discrete). Our conjectural answer to questions (1) and (2) is as follows:

\begin{conj}\label{conjecture-qualitative} If  \(\{\scheme W_k\}\) is a sequence of irreducible representations of \(\schG\) whose highest weight parameters grow to infinity inducing the same character \(\chi\colon \scheme Z_\Gamma \to \bbc^\times\), then \[\lim_{k \to \infty} \frac{\dim \HH^i(\Gamma, \scheme W_k)}{\dim \scheme W_k} = b_i^{\phi}(\Gamma)\,,\] where \(b_i^{\phi}(\Gamma)\) denotes the \(i\)-th \(\phi\)-Betti number of \(\Gamma\) for the induced character \(\phi = \operatorname{ind}^\Gamma_{\scheme Z_\Gamma}\chi\).
\end{conj}

As for question (3), we conjecture the following stronger statement:

\begin{conj}\label{conjecture-quantitative} If \(\lambda_i = \lambda_i(k)\) denote the highest weight parameters of the representations \(\scheme W_k\), then \begin{equation}\label{eq-normalized-cohomology-err}
    \left|\frac{\dim \HH^i(\Gamma, \scheme W_k)}{\dim \scheme W_k} - b_i^{\phi}(\Gamma)\right| = O\left(\frac{1}{\min \{\lambda_1(k),\ldots,\lambda_n(k)\}}\right)\,.\tag{\(**\)}
\end{equation}
\end{conj}

As evidence for the conjectures, we prove them when all the irreducible factors of \(\schG\) are of type \(A_1\):

\begin{thm}\label{thm-cohomology-sl2} Let \(\schG\) be the quotient of a product of finitely many copies of \(\SL_2(\bbc)\) by a finite central subgroup and \(\Gamma\) be a subgroup of \(\schG\) of type \(FP_\infty\) over \(\bbc\). Then, Conjecture~\ref{conjecture-qualitative} holds for \(\Gamma\) and Conjecture~\ref{conjecture-quantitative} holds for some finite index subgroup of \(\Gamma\).
\end{thm}

Let us now describe how Theorem~\ref{thm-cohomology-sl2} can be applied in a topological and an arithmetical setting. We start by remarking that if \(M\) is a \(3\)-manifold, possibly non-compact, with or without boundary, admitting a complete hyperbolic metric with totally geodesic boundary, and \(\Gamma = \pi_1(M) \leq \scheme{PSL}_2(\bbc)\) is finitely generated, then \(\Gamma\) is actually of type \(F\) and hence our results apply. Using the calculations of \(\dim \HH^i(\Gamma, \Sym^\lambda \bbc^2)\) by P. Menal-Ferrer and J. Porti in \cite{menal-ferrer_twisted_2012}, we obtain the following corollary:

\begin{cor}\label{cor-1} The \(\ell^2\)-Betti numbers of a topologically finite \(3\)-manifold \(M\) with a complete hyperbolic metric are given by \[b_i^{(2)}(M) = \begin{cases}
    0,&\text{if }i\neq 1\,,\\
    -\chi(M),&\text{if }i = 1\,.
\end{cases}\]
In particular, if the volume of \(M\) is finite then \(M\) is \(\ell^2\)-acyclic.
\end{cor}

The \(\ell^2\)-Betti numbers of all compact \(3\)-manifolds were computed by J. Lott and W. Lück in \cite{luckL2TopologicalInvariants3manifolds1995}. Using the prime decomposition and combination theorems, the core of their calculations lies in the finite volume hyperbolic case. The vanishing of the \(\ell^2\)-Betti numbers of finite volume hyperbolic manifolds also follows from I. Agol's proof for the Virtual Fibering conjecture, along with known vanishing results for \(\ell^2\)-Betti numbers of groups that algebraically fiber (see \cite{Lck2021}).

Moving on to number theory, if \(\Gamma\) is a congruence subgroup of \(\prod_{i=1}^{r_1+2r_2} \SL_2(\bbc)\) for a number field \(F\) with \(r_1\) real places and \(r_2\) conjugate pairs of complex places, then the Eichler-Shimura isomorphism relates the homology groups \(\HH_{r_1+r_2}(\Gamma, \scheme W)\) with coefficients in an irreducible representation of highest weight \(\bm{\lambda} = (\lambda_i)\) with the space \(S_{\bm{k}}(\Gamma)\) of cusp forms of level \(\Gamma\) and weight \(\bm{k} = (k_i)\), where \(k_i = \lambda_i + 2\). 

Describing the asymptotic behaviour of \(\dim_{\bbc} S_{\bm{k}}(\Gamma)\) is a major problem in current number theory research (\cite{shimizuDiscontinuousGroupsOperating1963,savinLimitMultiplicitiesCusp1989, sarnakBoundsMultiplicitiesAutomorphic1991,calegariBoundsMultiplicitiesUnitary2009,finisCohomologyLatticesSL2010,marshallBoundsMultiplicitiesCohomological2012,fu}). Moreover, not much is known about the growth rate of this dimension when \(\bm{k}\) grows, except when \(F\) is a totally real field. The best known upper bound for \(S_{\bm{k}}(\Gamma)\) is due to W. Fu:

\begin{thm}[{cf. \cite{fu}}]\label{cor-2} If \(F\) is not totally real, then \[\dim S_{\bm{k}}(\Gamma) = O\left(\frac{\prod_{i=1}^{r_1+2r_2} k_i}{\min k_i}\right)\,.\] In particular, \(\dim S_k(\Gamma) = O(k)\) if \(F\) is imaginary quadratic.
\end{thm}

In his proof, Fu establishes the validity of Conjecture~\ref{conjecture-qualitative} when \(\Gamma\) is a congruence subgroup. The way this is done is as follows: a spectral sequence of Emerton and Calegari relates the homology of \(\Gamma\) with coefficients in \(\Wmodk{}\) with the homology of a \(p\)-adic Lie group with coefficients in the ``completed homology'' of \(\Gamma\) as constructed by Emerton tensored with \(\Wmodk{}\). Since this tensored completed homology is a torsion module over the associated completed \(p\)-adic universal enveloping algebra \(\widehat{U(\ffg)}\), the desired polynomial bound follows an earlier theorem on torsion \(\widehat{U(\ffg)}\)-modules (\cite[Thm 3.2]{fu}).

Observe that Theorem~\ref{cor-2} also follows from Theorem~\ref{thm-cohomology-sl2} by the vanishing of the \(\ell^2\)-Betti numbers of lattices in \(\SL_2(\bbc)\), see Section~\ref{sec-cusp-forms}. It is also important to stress that \cite[Thm 3.2]{fu} is a critical ingredient in our proof of Theorem~\ref{thm-cohomology-sl2}; in fact one can see our result as an extension of Fu's theorem to general \(FP_\infty\) subgroups \(\Gamma\) of \(\schG\), not only the congruence ones, and this is achieved by avoiding Emerton's completed homology entirely.

Let us describe the structure of the paper and give an outline of the proof of Theorem~\ref{thm-cohomology-sl2}. 

There is a stronger way of formulating both Conjectures~\ref{conjecture-qualitative} and~\ref{conjecture-quantitative} in terms of Sylvester rank functions for matrices over \(\Gamma\), which are defined in Section~\ref{sec-liemodules}. We will prove Theorem~\ref{thm-cohomology-sl2} by actually proving the equivalent statement for these stronger Conjectures~\ref{lie-luck-conj} and~\ref{lie-luck-conj-quant} in Theorem~\ref{our-theorem-sl2}. We also delegate to Section~\ref{sec-vonNeumann} the study of Kionke's \(\phi\)-Betti numbers, where we recall some of its properties and establish new ones. 

After they are defined, we explore in Section~\ref{sec-Betti-numbers} how these rank functions can be used to compute \(\phi\)-Betti numbers, establishing the fact that Conjecture~\ref{lie-luck-conj} implies Conjecture~\ref{conjecture-qualitative}, and that Conjecture~\ref{lie-luck-conj-quant} implies Conjecture~\ref{conjecture-quantitative}.

Our main tool for attacking Theorem~\ref{thm-cohomology-sl2} is converting its statement to another one about representations of \(p\)-adic analytic pro-\(p\) groups. We dedicate Section~\ref{sec-analytic} to collecting the necessary definitions and results about \(p\)-adic analytic pro-\(p\) groups that will be needed throughout the paper.

This conversion is the objective of Section~\ref{sec-pro-p}, and its result presented in the statement of Corollary~\ref{thm-reduction-padic}. Our methods resemble those of D. Farkas and P. Linnell in \cite{farkasCongruenceSubgroupsAtiyah2006}. First, it is shown that if \(\Gamma\) virtually satisfies the quantitative Conjecture~\ref{lie-luck-conj-quant}, then it also satisfies the qualitative Conjecture~\ref{lie-luck-conj}. This step is actually the consequence of a general fact about rank functions defined on von Neumann factors, and it is carried out in Section~\ref{sec-crossed-products} (see Corollary~\ref{cor-simplified-conjecture}).

The second step is to isolate in \(\Gamma\) the finite index subgroup \(\Gamma'\) for which we are going to prove Conjecture~\ref{lie-luck-conj-quant}. This is done by embedding \(\Gamma\) into a \(p\)-adic analytic group (namely, the \(\bbz_p\)-points \(\schG(\bbz_p)\) of the algebraic group \(\schG\)) and taking the intersection of \(\Gamma\) with a uniform open normal subgroup. Since rank functions are invariant when passing to overgroups, this allows us to focus exclusively on the first congruence subgroup \(U_1\) of \(\schG(\bbz_p)\). Moreover, we also show that it suffices to study the asymptotic behaviour of the ranks of elements in the completed group algebra \(\Qalgebra{U_1}\) instead of general matrices.

So far, all of these facts are proven in complete generality, i.e. for an arbitrary semi-simple algebraic group \(\schG\) over \(\bbc\). It is in Section~\ref{sec-SL2} that we first specialize to the case in Theorem~\ref{thm-cohomology-sl2}, i.e. when all the irreducible factors of \(\schG\) have type \(A_1\). There, using the results of \cite{fu} and \cite{ardakov_irreducible_2013}, we prove Theorem~\ref{our-theorem-sl2}, the stronger version of Theorem~\ref{thm-cohomology-sl2}, by obtaining explicit bounds on the ranks of elements in \(\Qalgebra{U_1}\) for all odd primes \(p\).

The next two sections are dedicated to the applications of Theorem~\ref{thm-cohomology-sl2} mentioned in the introduction. In Section~\ref{sec-manifolds} we recall some essential facts about \(3\)-manifold topology that allows us to use Theorem~\ref{thm-cohomology-sl2} in order to compute the \(\ell^2\)-Betti numbers of topologically finite hyperbolic \(3\)-manifolds as in Corollary~\ref{cor-1}. Finally, in Section~\ref{sec-cusp-forms}, we briefly recall the Eichler-Shimura isomorphism, which relates the space of cusp forms to the cohomology of congruence subgroups. We then proceed to show how Theorem~\ref{thm-cohomology-sl2} implies Corollary~\ref{cor-2}.

The last two sections define the \(\phi\)-Betti numbers and prove some of its properties that are used throughout the paper. While we couldn't find any reference for some of those results, we would believe they are known to the specialist, and their proofs follow closely the arguments of the equivalent statements for the usual group von Neumann algebra.


\section*{Funding}

The first author was partially supported by the Basque Government's project IT483-22, the Spanish Government's project PID2020-117281GB-I00 and the University of the Basque Country's predoctoral fellowship PIF19/44. This work was partially supported by the grant PID2020-114032GB-I00/AEI/10.13039/501100011033.

\section*{Acknowledgements}

The authors are grateful to Andrei Jaikin-Zapirain for introducing them to the problem at the center of the Conjectures~\ref{conjecture-qualitative} and~\ref{conjecture-quantitative}, and for the suggestion of rephrasing this problem in terms of rank functions - as well as providing many useful comments and insights. The authors would also like to thank Oihana Garaialde Ocaña, Jon González-Sánchez, Leo Margolis, Simon Marshall, Joan Porti, Alan Reid, Pablo Sánchez-Peralta and Jan Boschheidgen for helpful discussions. Finally, the authors would like to thank the anonymous referees for their suggestions and comments.

\section{Lie modules and Sylvester functions}\label{sec-liemodules}

The objective of this section is to define the concepts necessary to state a stronger version of Conjectures~\ref{conjecture-qualitative} and~\ref{conjecture-quantitative}: the Sylvester rank and dimension functions associated to a finite dimensional representation of a group. We postpone to Section~\ref{sec-Betti-numbers} the proof that the conjectures stated here are indeed stronger than their respective counterparts in the introduction.

A Sylvester module rank function on a ring \(R\) is a function defined on the class of all finitely presented \(R\)-modules taking values on the non-negative reals that satisfies a series of conditions analogous to that of the usual dimension of vector spaces over a field (cf. Definition~\ref{defn-sylvester-module} below). Analogously, a Sylvester matrix rank function assigns a non-negative real number to each matrix over \(R\) in a fashion similar to the usual rank of matrices over fields (cf. Definition~\ref{defn-sylvester-matrix}). The viewpoint of Sylvester functions not only allows us to explain the asymptotic behaviour of cohomology in a more uniform fashion, but also permit us to relax the hypothesis of \(FP_\infty\) and focus instead on arbitrary finitely generated subgroups - again, not necessarily discrete.

For this section, fix \(\schG\) a semi-simple algebraic group over \(\bbc\) and \(\Gamma\) a finitely generated subgroup of \(\schG\). Henceforth, we assume \(\schG\) to be simply-connected (this will soon be made clear not to imply any loss of generality, see Remark~\ref{rmk-simply-connected}). We simultaneously see \(\schG\) as an algebraic group, a linear group over \(\bbc\) and an affine group scheme defined over \(\bbz\) by the associated Chevalley group, with context or an explicit reference making it clear which structure we are considering.

Since \(\schG\) is linear over \(\bbc\), by looking at the coefficients of the generators of \(\Gamma\) one can find a subring \(S_0\) of \(\bbc\) with field of fractions \(L_0 \leq \bbc\) such that \(\Gamma\) is a subgroup of the group \(\schG(S_0)\) of \(S_0\)-points of \(\schG\), where the notion of \(S_0\)-points in the linear and algebraic sense coincide by the simply-connected hypothesis \cite[Cor. 3 of Thm. 7]{steinberg_lectures_2016}.

Observe that every rational representation of \(\schG\) over \(\bbc\) is also defined over the prime field, i.e. it is the extension of scalars of a rational representation \(\schG(\bbq) \to \GL_{n}(\bbq)\) - cf. \cite[Cor. 1 to Thm. 7, Thm. 39(e) and the remark thereafter]{steinberg_lectures_2016}. Fix a maximal torus \(\scheme T\) in \(\schG\) and a set of positive roots for its root system \(R\). The irreducible rational representations of \(\schG\) are parameterized by the dominant integral elements \(\bm{\lambda}\) in the character group \(\scheme X(\scheme T)\) by taking the highest weight appearing in the weight decomposition. Moreover, \(\bm{\lambda}\) can be further parameterized by \(n\) non-negative integers \((\lambda_1,\ldots,\lambda_n)\) once an order on the roots is fixed. If \(\Wmodk{}\) is the irreducible left \(\schG\)-module associated to \(\bm\lambda\), one can find a \emph{rational form} (ie. a \(\bbq\)-submodule) \(\Wmodk{\bbq}\) in \(\Wmodk{}\) that is equivariant under \(\schG(\bbq)\) and such that \(\Wmodk{\bbq} \otimes_{\bbq} \bbc \simeq \Wmodk{}\). 

The uniqueness of the Chevalley group associated to a given weight lattice \cite[Cor. 5 to Thm. 4']{steinberg_lectures_2016} shows that the \(\schG(\bbq)\)-isomorphism class of \(\Wmodk{\bbq}\) does not depend on the chosen rational form. Hence, for any subfield \(L\) of \(\bbc\) we denote by \(\Wmodk{L}\) the \(\schG(L)\)-module obtained by the extension of scalars of \(\Wmodk{\bbq}\), which is well defined up to a rational isomorphism. In particular, we have that \(\Wmodk{\bbc} = \Wmodk{}\).

For each field \(L\) containing \(L_0\) and each finitely generated right \(L[\Gamma]\)-module \(M\), define \[\dim^\Gamma_{\Wmodk{L}} M = \frac{\dim_L M \otimes_{L[\Gamma]} \Wmodk{L}}{\dim_L \Wmodk{L}}\,.\] If \(L = \bbc\), we denote this function simply by \(\dim^\Gamma_{\Wmodk{}}\).

\begin{defn}\label{defn-sylvester-module} Given a unital ring \(R\), a function \(\dim\) on finitely presented \(R\)-modules valued on the non-negative reals is called a \emph{Sylvester module dimension function} if it satisfies the following properties:
\begin{enumerate}
	\item (SMod1) \(\dim 0 = 0\) and \(\dim R = 1\);
	\item (SMod2) \(\dim M_1 \oplus M_2 = \dim M_1 + \dim M_2\); and
	\item (SMod3) if \(M_1 \to M_2 \to M_3 \to 0\) is exact, then \[\dim M_3 \leq \dim M_2 \leq \dim M_1 + \dim M_3\,.\]
\end{enumerate}
\end{defn}

The function \(\dim^\Gamma_{\Wmodk{L}}\) satisfies (SMod1)-(SMod3), and is thus a Sylvester module dimension function on \(L[\Gamma]\).

We define the associated matrix rank function analogously. If \(A\) is an \(r\times s\) matrix over \(L[\Gamma]\) and \(M = L[\Gamma]^r/AL[\Gamma]^s\), then \begin{equation}\label{dualityrankdimension}\rk^\Gamma_{\Wmodk{L}} A = r - \dim_{\Wmodk{L}}^\Gamma M\,.\end{equation} Once more, if \(L = \bbc\) we denote this rank function simply by \(\rk^\Gamma_{\Wmodk{}}\).

\begin{defn}\label{defn-sylvester-matrix} A function \(\rk\) on matrices over \(R\) valued on the non-negative real numbers is called a \emph{Sylvester matrix rank function} if it satisfies the following properties:
\begin{enumerate}
    \item (SMat1) \(\rk A = 0\) if \(A\) is any zero matrix, and \(\rk(1) = 1\);
    \item (SMat2) \(\rk AB \leq \min\{\rk A\,,\,\rk B\}\) for any pair of matrices \(A\) and \(B\) that can be multiplied;
    \item(SMat3) \(\rk\begin{pmatrix} A & 0\\ 0 & B\end{pmatrix} = \rk A + \rk B\) for any matrices \(A\) and \(B\); and
    \item (SMat4) \(\rk\begin{pmatrix} A & C\\ 0 & B\end{pmatrix} \geq \rk A + \rk B\) for any matrices \(A\), \(B\) and \(C\) of appropriate sizes.
\end{enumerate}
\end{defn}

There is a general relation of the form given by equation~(\ref{dualityrankdimension}) that gives a bijection between the set of Sylvester matrix rank functions and Sylvester module dimension functions over any ring (\cite[Sec. 5]{jaikin-zapirainL2BettiNumbersTheir2019}). In particular, the function \(\rk = \rk^\Gamma_{\Wmodk{L}}\) satisfies  (SMat1)-(SMat4) and is thus a Sylvester matrix rank function on \(L[\Gamma]\).

Other examples of Sylvester matrix rank functions on \(L[\Gamma]\) are the von Neumann rank \(\rk_{\Gamma}\), \(\phi\)-rank \(\rk_\phi\) of Section~\ref{sec-vonNeumann} where \(\chi\colon \scheme Z_\Gamma \to \bbc^\times\) is any central character and \(\phi= \operatorname{ind}^\Gamma_{\scheme Z_\Gamma} \chi\) and the von Neumann ranks \[\rk_{\Gamma/\Gamma'}^L A = \frac{\rk_{L} A_{\Gamma'}}{|\Gamma:\Gamma'|}\] induced by the finite index normal subgroups \(\Gamma'\) of \(\Gamma\), where \(A_{\Gamma'}\) denotes the image of \(A\) under the natural map \(L[\Gamma] \to L[\Gamma/\Gamma']\).

More generally, if \(\varphi\colon R \to R'\) is a homomorphism of rings and \(\rk\) is a Sylvester matrix rank function on \(R'\), then the pullback along \(\varphi\) defines a Sylvester matrix rank function on \(R\) by \(\varphi^\#\rk(A) = \rk(\varphi(A))\), whose associated Sylvester module dimension function \(\varphi^\#\dim\) is given by \(\varphi^\#\dim M = \dim M \otimes_R R'\). The rank functions \(\rk^\Gamma_{\Wmodk{L}}\), \(\rk_\phi\) and \(\rk_{\Gamma/\Gamma'}^L\) can be seen as the pullbacks of the rank functions defined on \(\End_L(\Wmodk{L})\), \(\mathcal{N}_\phi\) and \(L[\Gamma/\Gamma']\) respectively.

If \(N\) is a left \(L[\Gamma]\)-module, one also defines \[\dim_{\Wmodk{L}}^\Gamma N = \frac{\dim_L (\Wmodk{L})^* \otimes_{L[\Gamma]} N}{\dim_L \Wmodk{L}}\,,\] where \((\Wmodk{L})^* = \Hom_L(\Wmodk{L},L)\) has the induced right \(L[\Gamma]\)-module structure (recall that \(\Wmodk{L}\) is a left \(\schG\)-module). If \(N = L[\Gamma]^r/L[\Gamma]^sA\) for some \(s\times r\)-matrix \(A\) where we now see elements of \(L[\Gamma]\) as row-vectors, then \(\dim^\Gamma_{\Wmodk{L}} N\) is the \(L\)-dimension of \[\mathrm{coker} ((\Wmodk{L})^*)^s \overset{\cdot A}{\to} ((\Wmodk{L})^*)^r\,.\] This map is the dual of the map \((\Wmodk{L})^r \overset{A\cdot }{\to} (\Wmodk{L})^s\). If we define \(M = L[\Gamma]^s/AL[\Gamma]^r\) now considering column vectors, the rank-nullity theorem tells us that \[r - \dim_{\Wmodk{L}}^\Gamma N = r - \dim(\ker ((\Wmodk{L})^r \overset{A\cdot}{\to} (\Wmodk{L})^s)) = \dim A(\Wmodk{L})^r = s-\dim^\Gamma_{\Wmodk{L}} M\,.\] Hence, it induces the same Sylvester matrix rank function \(\rk^\Gamma_{\Wmodk{L}}\) on matrices over \(L[\Gamma]\).

\begin{rmk}\label{remark-dim-extensions} For a matrix \(A \in \Mat(L[\Gamma])\) defined over a field \(L \geq L_0\) and any field extension \(L \leq K\), one also has that \(A \in \Mat(K[\Gamma])\). Observe that the rank is independent of the choice of field, i.e. \[\rk_{\Wmodk{K}}^\Gamma A = \rk_{\Wmodk{L}}^{\Gamma} A\,.\] This is because \[\dim^{\Gamma}_{\Wmodk{K}} K \otimes_L M = \dim^{\Gamma}_{\Wmodk{L}} M\,.\] In particular, \(\rk_{\Wmodk{L}}^\Gamma\) is induced by the composition  \(L[\Gamma] \to \End_L(\Wmodk{L}) \to \End_\bbc(\Wmodk{})\).
\end{rmk}

Let us now state a stronger formulation of Conjectures~\ref{conjecture-qualitative} and~\ref{conjecture-quantitative} in terms of Sylvester matrix rank functions. We recall our running hypothesis that \(\schG\) is a semi-simple algebraic group over \(\bbc\) and \(\Gamma\) is a finitely generated subgroup of \(\schG(\bbc)\). We denote the intersection \(\Gamma \cap \scheme Z(\schG)\) by \(\scheme Z_\Gamma\).

Given a central character \(\chi\colon \scheme Z (\schG) \to \bbc^\times\), denote by \(\scheme X(\scheme T)_\chi\) be the subset of the dominant integral elements \(\bm{\lambda} \in \scheme X(\scheme T)\) such that \(\scheme Z (\schG)\) acts on \(\Wmodk{}\) through \(\chi\). We believe that the general qualitative behaviour of the rank functions is described as follows:

\begin{conj}\label{lie-luck-conj} As rank functions on \(\bbc[\Gamma]\), we have that
\[\lim_{\min \lambda_i \to \infty}\rk^\Gamma_{\Wmodk{}} = \rk_\phi\,,\] where the limit is taken over \(\scheme X(\scheme T)_\chi\) and \(\rk_\phi\) is the \(\phi\)-rank for the induced character \(\phi = \operatorname{ind}^\Gamma_{\scheme Z_\Gamma} \chi\).
\end{conj}

As for the quantitative approach, we expect the following:

\begin{conj}\label{lie-luck-conj-quant} For any \(\bm \lambda \in \scheme X(\scheme T)_\chi\) and any matrix \(A\) over \(\bbc[\Gamma]\), we have that \[\left|\rk^\Gamma_{\Wmodk{}}(A) - \rk_\phi(A)\right| = O\left(\frac{1}{\min\{\lambda_1,\ldots,\lambda_n\}}\right)\,,\] where the implied constant depends on \(A\).
\end{conj}

\begin{rmk}\label{rmk-simply-connected} Let \(\scheme H\) be a connected semi-simple algebraic group over \(\bbc\) and \(\schG\) be the simply connected algebraic group with the same root system of \(\scheme H\). Then, if Conjecture~\ref{lie-luck-conj} holds for any finitely generated subgroup of \(\schG\), it also does for \(\scheme H\). Indeed, any matrix over \(\scheme H\) can be lifted to a matrix over \(\schG\), and one readily checks that both the ranks induced by a representation \(\scheme W\) and the twisted von Neumann ranks over \(\schG\) and over \(\scheme H\) coincide for this pair of matrices (see Proposition~\ref{prop:lifting-matrices}). The same is true for Conjecture~\ref{lie-luck-conj-quant}: if it holds (virtually) for all finitely generated subgroups of \(\schG\), it also holds (virtually) for all finitely generated subgroups of \(\scheme H\). Hence, it suffices to prove these conjectures for the unique simply connected semi-simple algebraic group of any given root system.
\end{rmk}

\section{\texorpdfstring{\(\ell^2\)}{l²}-Betti numbers}\label{sec-Betti-numbers}

Associated to any abstract group \(\Gamma\), we have a sequence of numerical invariants \(b_i^{(2)}(\Gamma)\) called the \emph{\(\ell^2\)-Betti numbers} of \(\Gamma\). They are defined as the von Neumann dimension of the \(\ell^2\)-homology groups \(\HH_i^{(2)}(\Gamma) = \HH_i(\Gamma,\mathcal{N}(\Gamma))\), where \(\mathcal{N}(\Gamma)\) is the von Neumann algebra generated by the action of \(\Gamma\) on \(\ell^2(\Gamma)\) \cite[Sec. 6.5]{luck_l2-invariants_2002}. If \(\Gamma\) is of type \(FP_{n+1}\) over \(\bbc\), the \(n\) first \(\ell^2(\Gamma)\)-Betti numbers of \(\Gamma\) can also be recovered from the complex \(F_\bullet \otimes_{\bbc[\Gamma]} \ell^2(\Gamma)\), where \(F_\bullet \to \bbc\) is a projective resolution of \(\bbc[\Gamma]\) with \(F_i\) of finite rank for \(0 \leq i \leq n+1\), by \cite[Lem. 6.53]{luck_l2-invariants_2002}. By taking instead the \(\phi\)-dimension \(\dim_\phi\) of \(\HH_i(\Gamma,\mathcal{N}(\Gamma)_\phi)\) for \(\phi = \operatorname{ind}^{\Gamma}_{\scheme Z_\Gamma} \chi\) as in Section~\ref{sec-vonNeumann}, one obtains the \(i\)-th \(\phi\)-Betti number of \(\Gamma\), and these all coincide with the usual \(\ell^2\)-Betti numbers if \(\scheme Z_\Gamma = \{1\}\) (see Remark~\ref{rmk-twisted-torsion-free}).

These \(\ell^2 \)-Betti numbers vanish if \(\Gamma\) is locally infinite amenable \cite[Thm. 7.2]{luck_l2-invariants_2002}, and the vanishing of \(\beta_1^{(2)}(\Gamma)\) is related to fibering in virtually residually finite rationally solvable (RFRS) groups \cite{kielak_residually_2020}, a class that includes all virtually compact special groups and hence the fundamental groups of finite volume hyperbolic \(3\)-manifolds.

The main objective of this section is to show how one can compute \(\phi\)-Betti numbers using the von Neumann rank for matrices, and hence to prove that the conjectures of Section~\ref{sec-liemodules} on Sylvester matrix rank functions imply their respective cohomological counterparts of the introduction.

A consequence of the (twisted) Lück's approximation theorem is that, for any finite-dimensional \(\bbc[\Gamma]\)-module \(\scheme W\) and any residual chain \(\{\Gamma_i\}\) of finite index normal subgroups of \(\Gamma\), the \(\ell^2\)-Betti number \(\beta_i^{(2)}(\Gamma)\) can be recovered as the homological gradient \[\beta_i^{(2)}(\Gamma)\cdot \dim_{\bbc} \scheme W = \lim_{n \to \infty} \frac{\dim_\bbc \HH_i(\Gamma_n, \scheme W)}{|\Gamma: \Gamma_n|}\,,\] as long as \(\Gamma\) is of type \(FP_{i+1}\) over \(\bbc\) \cite{boschheidgenTwistedL2Betti} (see also \cite[Thm. 13.49]{luck_l2-invariants_2002} for the case where \(\scheme W = \bbc\) is the trivial \(\Gamma\)-module). If \(\Gamma\) is a subgroup of some \(\schG\) satisfying  Conjecture~\ref{lie-luck-conj}, then the \(\bbc[\Gamma]\)-modules \(\Wmodk{}\) provide another way to approximate the \(\ell^2\)-analogues of Betti numbers of \(\Gamma\).

\begin{prop}\label{prop-homology-approx} Assume that \(\Gamma \leq \schG\) is a subgroup of type \(FP_{n+1}\) over \(\bbc\) that satisfies Conjecture~\ref{lie-luck-conj}. Then for any sequence \(\{\scheme W_k\}\) of irreducible representations of \(\schG\) with the same central character \(\chi\colon \scheme Z_\Gamma \to \bbc^\times\) such that \(\min\{\lambda_1(k),\ldots,\lambda_n(k)\} \to \infty\) we have that \[\lim_{k \to \infty} \frac{\dim_\bbc \HH_i(\Gamma, \scheme W_k)}{\dim_\bbc \scheme W_k} = \beta_i^{\phi}(\Gamma)\] whenever \(0 \leq i \leq n\) for \(\phi = \operatorname{ind}^\Gamma_{\scheme Z_\Gamma}\chi\). In particular, if \(\Gamma\) is of type \(FP_\infty\) over \(\bbc\), it also satisfies Conjecture~\ref{conjecture-qualitative}.
\end{prop}
\begin{proof} 
Let \(F_\bullet \to \bbc\) be a free resolution of \(\bbc\) over \(\bbc[\Gamma]\) such that \(F_i\) is of finite rank for every \(0 \leq i \leq k+1\). For every \(1 \leq i \leq k\), there exists a \(s\times r\) matrix \(A\) and a \(r \times t\) matrix \(B\) over \(\bbc[\Gamma]\) such that, after choosing bases for \(F_{i-1}\), \(F_i\) and \(F_{i+1}\), the maps \(\psi\) and \(\varphi\) in
\[\cdots \to F_{i+1} \overset{\psi}{\to} F_i \overset{\varphi}{\to} F_{i-1} \to \cdots\] are given by right multiplication with \(A\) and \(B\), respectively. By \cite[Lem. 2.4, Def. 2.11]{Kionke2018}, we have \[\beta_i^{\phi}(\Gamma) = r - \rk_\phi A - \rk_\phi B\,.\]



On the other hand, tensoring with \(\Wmodk{}\) over \(\bbc[\Gamma]\) gives us the complex \[\cdots \to \Wmodk{s} \overset{\widehat{\psi}}{\to} \Wmodk{r} \overset{\widehat{\varphi}}{\to} \Wmodk{t} \to \cdots\,,\] and \(\HH_i(\Gamma,\Wmodk{}) = \Tor_i^{\bbc[\Gamma]}(\bbc,\Wmodk{})\) is isomorphic to \(\ker(\widehat{\varphi})/\Image(\widehat{\psi})\). From this, we get that \[\dim_\bbc \HH_i(\Gamma,\Wmodk{}) = \dim_\bbc \ker(\widehat{\varphi})/\Image(\widehat{\psi}) = \dim_\bbc \ker(\widehat{\varphi}) - \dim_\bbc \Image(\widehat{\psi})\,.\]
Since \(\dim_\bbc \Image(\widehat{\psi}) = \dim_\bbc \Wmodk{} \cdot \rk^\Gamma_{\Wmodk{}} A\) and \(\dim_\bbc \ker(\widehat{\varphi}) = \dim_\bbc \Wmodk{} \cdot (r - \rk^\Gamma_{\Wmodk{}} B)\), we conclude that \[\dim_\bbc \HH_i(\Gamma,\Wmodk{}) = \dim_\bbc \Wmodk{}\cdot(r - \rk^\Gamma_{\Wmodk{}} A - \rk^\Gamma_{\Wmodk{}} B)\,.\] The validity of Conjecture~\ref{lie-luck-conj} now implies that \[\lim_{k \to \infty} \frac{\dim_\bbc \HH_i(\Gamma,\scheme W_k)}{\dim_\bbc W_k} = \lim_{k \to \infty} r- \rk^\Gamma_{\scheme W_{\bm \lambda(k)}} A - \rk^\Gamma_{\scheme W_{\bm\lambda(k)}}B = r - \rk_\phi A -\rk_\phi  B = \beta_i^{\phi}(\Gamma)\,.\]

To obtain the equivalent statement about cohomology, observe that \(\HH^i(\Gamma,\Wmodk{})\) can be computed by, instead of tensoring \(F_\bullet\) with \(\Wmodk{}\), taking the homology of the complex \(\Hom_{\bbc[\Gamma]}(F_\bullet,\Wmodk{})\). The tensor-hom adjunction formula tells us that this complex is isomorphic to \(\Hom_{\bbc}(F_\bullet \otimes_{\Gamma} (\Wmodk{})^*, \bbc)\), where the left \(\bbc[\Gamma]\)-module structure on \((\Wmodk{})^* = \Hom_\bbc(\Wmodk{},\bbc)\) is given by \((gf)(x) = f(g^{-1}x)\). This identifies \(\HH^i(\Gamma,\Wmodk{})\) with the complex dual of \(\HH_i(\Gamma,(\Wmodk{})^*)\).

Taking the complex dual of an irreducible \(\schG\)-representation corresponds to multiplying by \(-1\) its weight spectrum. Hence, \((\scheme W_{\bm \lambda})^* \simeq \scheme W_{\bm \lambda'}\) as left \(\schG\)-modules with highest weight \(\bm \lambda' = (\lambda'_1(k),\ldots,\lambda'_n(k)) = -w_0\bm\lambda\), where \(w_0\) is the longest element of the Weyl group of \(\schG\) (\cite[Prop. VII.7.11]{bourbakiLieGroupsLie2004}). Observe that the \(\lambda'_i\) are some permutation of the \(\lambda_j\), see the plates I-IX in \cite{bourbakiLieGroupsLie2002} which determine the block of \(w_0\) corresponding to each irreducible factor of the root system. In particular, if \(\min\{\lambda_1(k),\ldots,\lambda_n(k)\} \to \infty\), we also have \(\min\{\lambda_1'(k),\ldots,\lambda_n'(k)\} \to \infty\). Therefore:
\begin{align*}
    \lim_{k \to \infty} \frac{\dim_\bbc H^i(\Gamma,\scheme W_k)}{\dim_\bbc \scheme W_k} &= \lim_{k \to \infty} \frac{\dim_\bbc H_i(\Gamma, (\scheme W_k)^*)}{\dim_\bbc \scheme W_k}\\
    &= \lim_{k \to \infty} r - \rk^\Gamma_{\scheme W_{\bm\lambda'(k)}} A - \rk^\Gamma_{\scheme W_{\bm\lambda'(k)}} B\\
    &= r - \rk_\phi A - \rk_\phi B = \beta_i^{\phi}(\Gamma)\,.\qedhere
\end{align*}
\end{proof}

\begin{prop} Assume now that \(\Gamma \leq \schG\) is a subgroup of type \(FP_{n+1}\) over \(\bbc\) that satisfies Conjecture~\ref{lie-luck-conj-quant}. Then for any representation \(\scheme W\) of \(\schG\) with highest weight \(\bm\lambda = (\lambda_1,\ldots,\lambda_n)\) and central character \(\chi\colon \scheme Z_\Gamma \to \bbc^\times\) we have that \[\left|\frac{\dim_\bbc \HH_i(\Gamma,\scheme W)}{\dim_\bbc \scheme W} -  b_i^{\phi}(\Gamma)\right| = O\left(\frac{1}{\min\{\lambda_1,\ldots,\lambda_n\}}\right)\] for any \(0 \leq i \leq n\) where \(\phi=\operatorname{ind}^\Gamma_{\scheme Z_\Gamma} \chi\). In particular, if \(\Gamma\) is of type \(FP_\infty\), it also satisfies Conjecture~\ref{conjecture-quantitative}.
\end{prop}
\begin{proof}During the proof of Proposition~\ref{prop-homology-approx}, we show that, under these hypotheses, we have an upper bound
\[\left|\frac{\dim_{\bbc} \HH_i(\Gamma,\Wmodk{})}{\dim_{\bbc} \Wmodk{}} - \beta_i^{\phi}(\Gamma)\right| \leq \left|\rk^\Gamma_{\Wmodk{}}(A) - \rk_\phi(A)\right| + \left|\rk^\Gamma_{\Wmodk{}}(B) - \rk_\phi(B)\right|\] where \(A\) and \(B\) are matrices over \(\bbc[\Gamma]\). So, once again, the result follows from the validity of Conjecture~\ref{lie-luck-conj-quant}. To obtain the equivalent statement about cohomology, note again that \(\HH^i(\Gamma,\scheme W)\) is isomorphic to the complex dual of \(\HH_i(\Gamma, \scheme W^*)\). If \(\bm\lambda = (\lambda_1,\cdots,\lambda_n)\) denotes the highest weight of \(\scheme W\), then the highest weight of \(\scheme W^*\) is given by \(-w_0\bm\lambda = (\lambda_1',\ldots,\lambda_n')\), where \(w_0\) denotes the longest element of the Weyl group. Once again, the \(\lambda_i'\) are some permutation of the \(\lambda_j\) (\cite[Plates I-IX]{bourbakiLieGroupsLie2002}) Hence: \[O\left(\frac{1}{\min\{\lambda_1,\ldots,\lambda_n\}}\right) = O\left(\frac{1}{\min\{\lambda_1',\ldots,\lambda_n'\}}\right)\,,\] and thus \begin{align*}\left|\frac{\dim_\bbc \HH^i(\Gamma,\scheme W)}{\dim_\bbc \scheme W} - \beta_i^{\phi}(\Gamma)\right| &= \left|\frac{\dim_\bbc \HH_i(\Gamma,\scheme W^*)}{\dim_\bbc \scheme W^*} - \beta_i^{\phi}(\Gamma)\right|\\
&= O\left(\frac{1}{\min\{\lambda_1',\ldots,\lambda_n'\}}\right)\\
&= O\left(\frac{1}{\min\{\lambda_1,\ldots,\lambda_n\}}\right)\qedhere
\end{align*}
\end{proof}

It is important to note that in many cases, satisfying Conjecture~\ref{conjecture-quantitative} only virtually is enough to obtain bounds on cohomology. We do it in the following proposition:

\begin{prop}\label{prop-0-betti-number} Assume \(\Gamma \leq \schG\) is a torsion-free subgroup of type \(FP_{n+1}\) that virtually satisfies Conjecture~\ref{lie-luck-conj-quant}. Then, if \(\beta_i^{(2)}(\Gamma) = 0\) for some \(0\leq i \leq n\), we have that
\[\frac{\dim_\bbc \HH^i(\Gamma, \scheme W)}{\dim_\bbc \scheme W} = O\left(\frac{1}{\min\{\lambda_1,\ldots,\lambda_n\}}\right)\,.\]
\end{prop}
\begin{proof} Let \(\Gamma'\) be a finite index subgroup of \(\Gamma\) that satisfies Conjecture~\ref{lie-luck-conj-quant}. Let \(\alpha\colon \HH_i(\Gamma', \scheme W) \to \HH_i(\Gamma, \scheme W)\) and \(\beta\colon \HH_i(\Gamma, \scheme W) \to \HH_i(\Gamma', \scheme W)\) be the corestriction and transfer maps in homology, respectively. Since \(\alpha\circ\beta\) equals multiplication by the index \(|\Gamma\colon\Gamma'|\) in \(\HH_i(\Gamma,\scheme W)\), it is an isomorphism and therefore \(\beta\) must be injective, see \cite[Prop. 9.5(ii)]{brown_cohomology_1982}. Consequently, \[\dim_{\bbc} \HH_i(\Gamma,\scheme W) \leq \dim_{\bbc} \HH_i(\Gamma',\scheme W)\] and, since \(\beta_i^{(2)}(\Gamma') =| \Gamma\colon \Gamma'|\beta_i^{(2)}(\Gamma) = 0\), the claim follows.
\end{proof}

\section{\texorpdfstring{\(p\)}{p}-adic analytic groups}\label{sec-analytic}

Our main tool for attacking Conjecture~\ref{lie-luck-conj-quant} is translating it into a statement about \(p\)-adic universal enveloping algebras. This is done in Section~\ref{sec-pro-p} by exploiting the rational structure of \(\schG\) and embedding the finitely generated subfield \(S_0 \leq \bbc\) into a \(p\)-adic field. By taking the closure with respect to the \(p\)-adic topology, we then shift the focus to asymptotic properties of \(p\)-adic analytic groups. Hence, this section is devoted to collecting the definitions, properties and theorems about \(p\)-adic analytic groups that will be needed for the rest of the paper.

Let \(G\) be a profinite group and \(p\) be and odd prime number. The completed group algebra of $G$, also referred to as the \emph{Iwasawa algebra} in the literature, is \[\Zalgebra{G} = \varprojlim \bbz_p[G/U]\,,\] where the inverse limit is taken over all the open normal subgroups \(U\) of \(G\).
For any pair of right and left profinite (i.e. compact, Hausdorff, totally disconnected modules with a continuous \(\Zalgebra{G}\)-action) \(\Zalgebra{G}\)-modules \(M\) and \(N\), we let \[M\otimes_{\Zalgebra{G}} N\quad\text{and}\quad M\cotimes_{\Zalgebra{G}} N\] be the usual and the completed tensor product over the ring \(\Zalgebra{G}\), respectively (see \cite[Sec. 5.5]{ribes_profinite_2010}). Both have the structure of a \(\bbz_p\)-module, with the latter being the completion of the former with respect to the topology defined by the finite quotients of \(M\) and \(N\). We note that, if either \(M\) or \(N\) is finitely generated, then \(M\cotimes_{\Zalgebra{G}} N \simeq M\otimes_{\Zalgebra{G}} N\) by \cite[Lem. 2.1(i)]{brumer_pseudocompact_1966}.

A profinite group \(G\) is called \emph{\(p\)-adic analytic} if it is a Lie group over \(\bbq_p\), 
see \cite[Interlude A]{dixon_analytic_1999} for the many equivalent definitions. We say that \(G\) has dimension \(d\) if it has dimension \(d\) as an analytic manifold; the integer $d$ can also be recovered as the dimension over \(\bbq_p\) of its Lie algebra and as the minimal number of topological generators for any open uniform subgroup \(U\) of \(G\).

Let \(\Qalgebra{G} = \Zalgebra{G} \otimes_{\bbz_p} \bbq_p\). Opposite to what the notation might suggest, the algebra \(\Qalgebra{G}\) is \emph{not} in general isomorphic to the inverse limit \(\varprojlim \bbq_p[G/U]\). The completed group algebra \(\Zalgebra{G}\) is known to be an Ore domain if \(G\) is a torsion-free \(p\)-adic analytic profinite (and hence pro-\(p\)) group (see \cite[Sec. 4]{ardakov_ring-theoretic_2006}). 
Let \(\mathcal{Q}_{G}\) be the associated division ring of fractions. For any left \(\Zalgebra{G}\)-module \(M\) and any right \(\Zalgebra{G}\)-module \(N\), we define \[\dim_{\Zalgebra{G}} M = \dim_{\mathcal{Q}_G} \mathcal{Q}_G \otimes_{\Zalgebra{G}} M\quad\text{and}\quad \dim_{\Zalgebra{G}} N = \dim_{\mathcal{Q}_G} N\otimes_{\Zalgebra{G}} \mathcal{Q}_G\,.\]
By the right exactness of the tensor product, one readily checks that this is a Sylvester module dimension function on \(\Zalgebra{G}\) whose associated matrix rank function is \(\rk_{\Zalgebra{G}} = \rk_{\mathcal{Q}_G}\). Observe that \(\mathcal{Q}_G\) also contains \(\Qalgebra{G}\), and therefore induces a matrix rank function \(\rk_{\Qalgebra{G}}\) on \(\Qalgebra{G}\) with associated dimension function \(\dim_{\Qalgebra{G}}\).

For any \(\bbz_p\)-module \(M\), we define \(\dim_{\bbz_p} M\) as \(\dim_{\bbq_p} \bbq_p \otimes_{\bbq_p} M\). The following result is an analogue of the Lück approximation theorem for analytic pro-\(p\) groups.

\begin{thm}[{\cite[Thm. 1.10]{harris_span_1979}, cf. \cite[Thm. 2.1]{bergeron_growth_2014}}]\label{thm-harris} Let \(G \leq \GL_n(\bbz_p)\) be a torsion-free \(p\)-adic analytic group of dimension \(d\). Let \(U_i\) be the intersection of \(G\) with the \(i\)-th congruence subgroup of \(\GL_n(\bbz_p)\). If \(M\) is a finitely generated \(\Zalgebra{G}\)-module, then \[\left|\dim_{\bbz_p} (\bbz_p \otimes_{\Zalgebra{U_i}} M) -\dim_{\Zalgebra{G}}(M)\cdot |G: U_i|\right|=O(|G: U_i|^{1-1/d})\,.\] In particular, we have that \[\lim_{i \to \infty} \frac{\dim_{\bbz_p} (\bbz_p \otimes_{\Zalgebra{U_i}} M)}{|G: U_i|} = \dim_{\Zalgebra{G}} M\,.\]
\end{thm}


From Harris' result we obtain the analogue statement for rank functions.

\begin{cor}\label{cor-Harris-approx} Let \(G\) be \(p\)-adic analytic torsion-free and \(A\) be a matrix over \(\Zalgebra{G}\). Then, \[\lim_{i \to \infty} \rk_{G/U_i}^{\bbz_p} A_i = \rk_{\Zalgebra{G}} A\,,\] where \(A_i\) denotes the image of \(A\) over \(\bbz_p[G/U_i]\).
\end{cor}


There is a correspondence between uniform pro-$p$ groups and uniform $\bbz_p$-Lie algebras, originally described by Lazard \cite{lazard_1965}, given by $\exp$ and $\log$ functors. Let $G$ be a uniform pro-$p$ group, and consider the formal power series
\begin{displaymath}
    \exp(X)=\sum_{i=0}^{\infty}\frac{X^i}{i!} \quad \text{and} \quad \log(X+1)=\sum_{i=1}^{\infty}(-1)^{i+1}\frac{X^i}{i}.
\end{displaymath}
Inside the completed group algebra \(\Qalgebra{G}\), the series $\log(g)$ converges for all $g\in G$, and $\log(G)$ is a $\bbz_p$-Lie algebra with the bracket given by the commutator of $\Qalgebra{G}$ (\cite[Thm. 7.13]{dixon_analytic_1999}). Furthermore, we have that $\exp(\log(G))=G$. More generally, if \(G\) is a subgroup of the group of units of a Banach \(\bbq_p\)-algebra \(A\), the topology on \(G\) is induced by the norm of \(A\) and all the elements \(g-1\) have norm strictly less than \(1\), we can identify \(\log(G) \subseteq A\) with \(\log(G) \subseteq \Qalgebra{G}\) as \(\bbz_p\)-Lie algebras \cite[Cor. 7.14]{dixon_analytic_1999}. We call such an algebra \(A\) an \emph{associative Banach envelope} of the uniform pro-\(p\) group \(G\).

This correspondence allows us to describe algebraically the $\bbq_p$-Lie algebra $\ffg$ associated to $G$ as a $p$-adic analytic group, which is nothing more than \(\ffg\simeq\bbq_p\otimes_{\bbz_p} \operatorname{log}(G)\) (\cite[Sec. 9.5]{dixon_analytic_1999}). The Lie algebra \(\operatorname{log}(G)\) is \emph{powerful} in the sense that \([\operatorname{log}(G),\operatorname{log}(G)] \subseteq p\operatorname{log}(G)\) (\cite[Exer. 4.2]{dixon_analytic_1999}). If \(G\) has dimension \(d\), then \(\operatorname{log}(G)\) is a free \(\bbz_p\)-module of rank \(d\) (\cite[Thm. 7.13]{dixon_analytic_1999}). Taking any $\bbz_p$-basis \(\{h_1,\cdots,h_d\}\) of \(\operatorname{log}(G)\), the Poincaré-Birkoff-Witt theorem tells us that the monomials \(h_1^{a_1}\cdots h_d^{a_d}\) form a \(\bbq_p\)-basis of the universal enveloping algebra \(U(\ffg)\) (\cite[Cor. I.2.7.2]{bourbakiLieGroupsLie1989}).


Define a norm on the \(\bbq_p\)-vector space \(U(\ffg)\) by setting \[\left\|\sum_{a_1,\ldots,a_d} \alpha_{a_1,\ldots,a_d}h_1^{a_1}\cdots h_d^{a_d}\right\| = \max\{|\alpha_{a_1,\ldots,a_d}|\}\,.\] Since \(\log(G)\) is a powerful Lie algebra, this norm is submultiplicative. By completing \(U(\ffg)\) we obtain a Banach \(\bbq_p\)-algebra $\widehat{U(\ffg)}$ called the \emph{completed universal enveloping algebra} of \(\ffg\), which can be described as
\[\widehat{U(\ffg)} = \left\{\sum_{a_1,\ldots,a_d} \alpha_{a_1,\ldots,a_d}h_1^{a_1}\cdots h_d^{a_d} : |\alpha_{a_1,\ldots,a_d}| \to 0\text{ as }(a_1,\ldots,a_d) \to \infty\right\}\,.\]
The completed universal enveloping algebra is an \emph{almost commutative affinoid algebra} in the sense of \cite[Sec. 3]{ardakov_irreducible_2013}. We collect some properties of the completed group algebra \(\Zalgebra{G}\) and the associated completed universal enveloping algebra in the proposition that follows. Recall that a subset \(T\) of regular elements in a ring \(R\) is an \emph{Ore set} if for any non-zero \(t \in T\) and \(x \in R\) we have \(tR\cap xT \neq \varnothing\) and \(Rt\cap Tx \neq \varnothing\). For any Ore set, one can form the \emph{Ore localization} \(R \to RT^{-1}\), which a ring homomorphism with the universal property of any element in \(T\) becoming invertible in \(RT^{-1}\) (see \cite[Chap. 10]{goodearlIntroductionNoncommutativeNoetherian2004}).

\begin{prop}\label{prop-embedding-enveloping-alg} Let \(G\), \(\log(G)\) and \(\ffg\) be as before. Take  an associative Banach envelope \(A\) for \(G\) with logarithm map \(\log\colon G \to A\). Let \(\iota\colon \ffg \to \widehat{U(\ffg)}\) be the inclusion map from \(\ffg\) into its completed universal enveloping algebra.
\begin{enumerate}
    \item The ring \(\Zalgebra{G}\) is a local ring and the augmentation ideal \(I\) is the unique maximal two-sided ideal.
    \item The set \[T = \bigcup_{a \geq 0} p^a + I^{a+1}\] is an Ore set in \(\Zalgebra{G}\).
    \item The composition
    \[G \overset{\log}{\to}  \ffg \overset{\iota}{\to} \widehat{U(\ffg)} \overset{\exp}{\to} \widehat{U(\ffg)}\] induces an isomorphism \(\widehat{\Zalgebra{G}T^{-1}} \simeq \widehat{U(\ffg)}\). In particular, \(\widehat{U(\ffg)}\) is flat as a left and as a right \(\Zalgebra{G}\)-module.
\end{enumerate}
\end{prop}
\begin{proof} (1) is generally true for completed group algebras of pro-\(p\) groups over DVRs with residue field of characteristic \(p\), see \cite[Sec. 4]{brumer_pseudocompact_1966}. (2) follows from \cite[Lem. 2.4]{ardakov_irreducible_2013}. (3) The isomorphism is shown in \cite[Thm. 10.4]{ardakov_irreducible_2013}.
\end{proof}

\begin{rmk}\label{rmk-profinite-points} Let us observe how this objects appear in the context of  Conjecture~\ref{lie-luck-conj}. The group \(\schG(\bbz_p)\) of \(\bbz_p\)-points of \(\schG\) is a \(p\)-adic analytic group and, if \(\ffh\) is the Chevalley \(\bbz\)-Lie algebra of \(\schG\), it is related to the analytic Lie algebra of \(G\) by an extension of scalars: \(\ffg \simeq \bbq_p \otimes_{\bbz} \ffh\).

Consider the congruence subgroups \(U_i = \ker(\schG(\bbz_p) \to \schG(\bbz_p/p^i\bbz_p))\). By \cite[Lem. 2.2.2]{huber_complements_2011} and the remark thereafter, for odd primes \(p\) the subgroups \(U_i\) are all uniform pro-\(p\) groups for \(i \geq 1\). In fact, \(U_1\) can be identified with J.-P. Serre's standard formal group associated to the group law defined by \(\schG\), and the \(U_i\) are given by the lower \(p\)-central series of \(U_1\). Observe that \(\ffg\) is the \(\bbq_p\)-Lie algebra of all the \(U_i\), since they are open in \(\schG(\bbz_p)\). To each \(U_i\), we can associate Lazard's \(\bbz_p\)-Lie algebra \(\log(U_i)\), and we have an isomorphism of \(\bbz_p\)-Lie algebras \(\log(U_i) \simeq p^i\bbz_p\otimes_{\bbz} \ffh\) by observing that \(\schG(\bbz_p)\) is linear over \(\bbq_p\) and thus \(U_i \to \scheme{Mat}_n(\bbq_p)\) is an associative Banach envelope.
\end{rmk}



If \(M\) is a \(U(\ffg)\)-module that is finite dimensional over \(\bbq_p\), then the action of \(U(\ffg)\) continuously extends to an action of \(\widehat{U(\ffg)}\).
Consider now the modules \(\Wmodk{\bbq_p}\) of Section~\ref{sec-liemodules}. This action also  defines a module \(\widetilde{\Wmodk{\bbq_p}}\) over \(\Qalgebra{U_i}\) by restriction. On the other hand, we have an action of the abstract group algebra \(\bbz_p[U_i]\) on \(\Wmodk{\bbq_p}\) 
which extends to an action of \(\Zalgebra{U_i}\) -- by the completeness of \(\Wmodk{\bbq_p}\) and the description of \(\Zalgebra{U_i}\) as a ring of non-commutative power series \cite[Thm. 7.20]{dixon_analytic_1999}. Hence, \(\Wmodk{\bbq_p}\) also has the structure of a \(\Qalgebra{U_i}\)-module.

\begin{lem}\label{lem-same-module} We have that \(\widetilde{\Wmodk{\bbq_p}} \simeq \Wmodk{\bbq_p}\) as \(\Qalgebra{U_i}\)-modules.
\end{lem}
\begin{proof} After identifying the underlying sets of \(\widetilde{\Wmodk{\bbq_p}}\) and \(\Wmodk{\bbq_p}\), the isomorphism is an immediate consequence of the Lie group/Lie algebra correspondence, which provides the outer square of a commutative diagram
\begin{center}
\begin{tikzcd}
U_i \arrow[rr] \arrow[dd, "\log"']       &                                    & \End_{\bbq_p}(\Wmodk{\bbq_p})                              \\
                                         & \widehat{U(\ffg)} \arrow[ru] &                                                \\
\ffg \arrow[rr] \arrow[ru, "\exp"] &                                    & \mathfrak{gl}_{\bbq_p}(\Wmodk{\bbq_p}) \arrow[uu, "\exp"']
\end{tikzcd}
\end{center}
whose lower triangle commutes by continuity.
\end{proof}

Now that there is no ambiguity about the \(\Qalgebra{U_i}\)-module structure of \(\Wmodk{\bbq_p}\), we conclude this section by remarking that, for every closed subgroup \(H\) of \(U_i\) and finitely generated \(\Qalgebra{H}\)-module \(M\), we have the Sylvester module dimension functions \[\dim^{H}_{\Wmodk{\bbq_p}} M= \frac{\dim_{\bbq_p} M \otimes_{\Qalgebra{H}} \Wmodk{\bbq_p}}{\dim_{\bbq_p} \Wmodk{\bbq_p}}\] with associated Sylvester matrix rank functions \(\rk^H_{\Wmodk{\bbq_p}}\).

\section{Going from \texorpdfstring{\(\bbc\)}{C} to \texorpdfstring{\(p\)}{p}-adic fields}\label{sec-pro-p}

We keep the notation from Section~\ref{sec-liemodules}: \(\schG\) is a semi-simple algebraic group over \(\bbc\), each dominant integral element \(\bm\lambda\) in the character group \(\scheme X(\scheme T)\) of a maximal torus \(\scheme T\) yields an irreducible rational representation \(\Wmodk{}\) defined over \(\bbq\) and \(\Gamma\) is a finitely generated subgroup of the group \(\schG(S_0)\) of \(S_0\)-points of \(\schG\), where \(S_0\) is a finitely generated subring of \(\bbc\) with field of fractions \(L_0\).

The first step in translating Conjecture~\ref{lie-luck-conj} to a \(p\)-adic analytic setting is the following result. By Cassels' embedding theorem \cite{casselsEmbeddingTheoremFields1976}, for each finitely generated subring \(S\) of \(\bbc\) there are infinitely many primes \(p\) such that \(S\) embeds into \(\bbz_p\). From this, we obtain the following result:

\begin{prop}\label{prop:embedding} For each finitely generated subring \(S\) of \(\bbc\) there are infinitely many primes \(p\) such that \(\schG(S)\) embeds in \(\schG(\bbz_p)\).
\end{prop}

We call a prime \(p\) for which Proposition~\ref{prop:embedding} holds an \emph{enveloping prime} of the ring \(S\). If \(A\) is a matrix over \(\bbc[\Gamma]\), then it can be seen as a matrix over \(S[\Gamma]\) for some finitely generated subring \(S\) of \(\bbc\) containing \(S_0\) with field of fractions \(L\). Observe that by the Remark~\ref{remark-dim-extensions}, we have that \[\rk^\Gamma_{\Wmodk{}} A = \rk^\Gamma_{\Wmodk{L}} A = \rk^\Gamma_{\Wmodk{\bbq_p}} A\] for any enveloping prime \(p\) of \(S\).

We now proceed to obtain the finite index subgroup \(\Gamma'\) for which we are going to try and prove Conjecture~\ref{lie-luck-conj-quant}. We recall the congruence subgroups \(U_i\) of \(\schG(\bbz_p)\) defined on Remark~\ref{rmk-profinite-points}, i.e. \(U_i = \ker(\schG(\bbz_p) \to \schG(\bbz_p/p^i\bbz_p))\). Recall that a finitely generated pro-\(p\) group \(V\) is uniform if \([V,V] \leq V^{\epsilon p}\) where \(\epsilon = 2\) if \(p = 2\) and \(\epsilon = 1\) otherwise, and \(|P_i(V)\colon P_{i+1}(V)| = |V\colon P_2(V)|\), where \(P_i(V)\) denotes the \(i\)-th term of the lower \(p\)-central series of \(V\). Moreover, any \(p\)-adic analytic group contains an open uniform subgroup (\cite[Thm. 8.32]{dixon_analytic_1999}).

\begin{lem}\label{lem-existence-down} For any enveloping prime \(p\) of \(S\), there exists a finite index normal subgroup \(\Gamma'\) of \(\Gamma\) such that its \(p\)-adic closure \(\overline{\Gamma'}\) is a uniform pro-\(p\) subgroup of \(U_1\).
\end{lem}
\begin{proof} The \(p\)-adic closure \(\overline{\Gamma}\) in \(\schG(\bbz_p)\) is \(p\)-adic analytic, and hence its intersection with \(U_1\) contains some open uniform normal subgroup \(V\). We take \(\Gamma' = \Gamma \cap V\).
\end{proof}

Let \(G = \overline{\Gamma'}\), where $\Gamma'$ is the subgroup from Lemma~\ref{lem-existence-down}. Observe that both \(G\) and \(\Gamma'\) are torsion-free. Hence, if \(\chi\) is any central character of \(\scheme{Z}(\schG)\) and \(\phi = \operatorname{ind}^{\Gamma'}_{\scheme Z(\schG) \cap \Gamma'} \chi\), the \(\phi\)-rank \(\rk_\phi\) coincides with the usual von Neumann rank \(\rk_{\Gamma'}\) by Remark~\ref{rmk-twisted-torsion-free}. Using Lück's approximation and Harris' theorem, we can compare \(\rk_{\Gamma'}\) to \(\rk_{\Zalgebra{G}}\).

\begin{prop}\label{prop-transfer-luck-harris} Let \(A\) be a matrix over \(S[\Gamma']\), where \(S_0 \leq S\) is a finitely generated subring of \(\bbc\). If \(p\) is an odd enveloping prime of \(S\), we have that \[\rk_{\Gamma'} A = \rk_{\Qalgebra{G}} A = \rk_{\Qalgebra{U_1}} A\,.\]
\end{prop}
\begin{proof} Let \(L\) be the field of fractions of \(S\), and take \(\Gamma_i' = \Gamma' \cap U_i\). Then, the \(\Gamma_i'\) constitute a chain of normal subgroups of \(\Gamma'\) of finite index with trivial intersection, and moreover \(\Gamma'/\Gamma_i'\simeq G/(G\cap U_i)\). By Lück's approximation theorem (Theorem~\ref{sofic-luck}, see \cite{jaikin-zapirain_base_2019, dodziukApproximatingL2invariantsAtiyah2003,luck_approximation_1994}), we have that 
\begin{align*}
    \rk_{\Gamma'} A &= \lim_{i \to \infty} \rk_{\Gamma'/\Gamma_i'}^\bbc A_i = \lim_{i \to \infty} \rk_{\Gamma'/\Gamma_i'}^L A_i \\&= \lim_{i \to \infty} \rk_{G/(G\cap U_i)}^{\bbq_p} A_i = \lim_{i \to \infty} \rk_{G/(G\cap U_i)}^{\bbz_p} A_i\,,
\end{align*}
where \(A_i\) denotes the image of \(A\) in \(S[\Gamma'/\Gamma_i'] \subseteq \bbz_p[G/(G\cap U_i)]\). Now, the claim follows from Corollary~\ref{cor-Harris-approx}, observing that the embedding \(\Zalgebra{G} \to \Zalgebra{U_1}\) induces an embedding of their respective division rings of fractions.
\end{proof}

It now only remains to reinterpret \(\rk_{\Wmodk{\bbq_p}}^{\Gamma'} A\) in terms of \(G\).

\begin{lem}\label{transfer-Lie-C-p-adic} For every matrix \(A\) over \(\bbq_p[\Gamma']\), we have that \[\rk_{\Wmodk{\bbq_p}}^{\Gamma'} A = \rk_{\Wmodk{\bbq_p}}^{G} A = \rk_{\Wmodk{\bbq_p}}^{U_1} A\,.\]
\end{lem}
\begin{proof} Let \(A \in \Mat_{n\times m}(\bbq_p[\Gamma]')\), \(M\) be the \(\bbq_p[\Gamma']\)-module \(\bbq_p[\Gamma']^n/A\bbq_p[\Gamma']^m\) and \(\widehat{M}\) be the \(\Qalgebra{G}\)-module \(\Qalgebra{G}^n/A\Qalgebra{G}^m\). It is immediately checked that \(\widehat{M} = M \otimes_{\bbq_p[\Gamma']} \Qalgebra{G}\). Furthermore, we have that 
\begin{align*}
    \dim_{\Wmodk{\bbq_p}}^{G} \widehat{M} &= \frac{\dim_{\bbq_p} \widehat{M} \otimes_{\Qalgebra{G}} \Wmodk{\bbq_p}}{\dim_{\bbq_p}\Wmodk{\bbq_p}}\\
    &= \frac{\dim_{\bbq_p} (M \otimes_{\bbq_p[\Gamma']} \Qalgebra{G}) \otimes_{\Qalgebra{G}} \Wmodk{\bbq_p}}{\dim_{\bbq_p} \Wmodk{\bbq_p}}\\
    &= \frac{\dim_{\bbq_p} M \otimes_{\bbq_p[\Gamma']} \Wmodk{\bbq_p}}{\dim_{\bbq_p} \Wmodk{\bbq_p}}\\
    &= \dim^{\Gamma'}_{\Wmodk{\bbq_p}} M.
\end{align*}
The first equality of matrix ranks then follows, and the second equality is proven analogously.
\end{proof}

We are now able to state and prove the connection between   Conjectures~\ref{lie-luck-conj} and~\ref{lie-luck-conj-quant} and the \(p\)-adic analytic setting.

\begin{thm}\label{thm-reduction-padic} Let \(\schG\) be a simply-connected and semi-simple algebraic group. Assume that, for any increasing sequence of weights \(\{\bm\lambda = \bm \lambda(k)\}\) such that \(\min \{\lambda_i(k)\}\) grows to infinity, we have that 
\begin{equation}
	\lim_{k \to \infty} \rk_{\Wmodk{\bbq_p}}^{U_1} = \rk_{\Qalgebra{U_1}}
	\label{eqn-lim-rk-cong}
\end{equation} as rank functions on \(\Qalgebra{U_1}\) for all but finitely many primes \(p\). Then, Conjecture~\ref{lie-luck-conj} holds for all finitely generated subgroups \(\Gamma\) of \(\schG\).
\end{thm}
\begin{proof} We will apply Corollary~\ref{cor-simplified-conjecture}. Let \(\Gamma\) be a finitely generated and Zariski-dense subgroup of \(\schG\) with \(\Gamma \leq \schG(S)\) where \(S\) is a finitely generated subring of \(\bbc\) with field of fractions \(L\). Proposition~\ref{prop:embedding} gives us infinitely many enveloping primes \(p\) for \(S\) such that \(\Gamma\) is a subgroup of \(\schG(\bbz_p)\). Then, by Lemma~\ref{lem-existence-down} we obtain a finite index normal subgroup \(\Gamma'\) of \(\Gamma\) such that \(\overline{\Gamma'} = G\) is an uniform subgroup of \(U_1\). By Proposition~\ref{prop-transfer-luck-harris} and Lemma~\ref{transfer-Lie-C-p-adic}, condition (\ref{eqn-lim-rk-cong}) implies that \[\lim_{k \to \infty} \rk_{\scheme{W}^L_{\bm\lambda_k}}^{\Gamma'} = \rk_{\Gamma'}\] on \(L[\Gamma']\). Hence, it suffices to choose an enveloping prime \(p\) for which (\ref{eqn-lim-rk-cong}) holds.
\end{proof}

\begin{rmk}\label{rmk-sylvester-regular} To apply Theorem~\ref{thm-reduction-padic} it suffices to show that, for every non-zero \(x\) in \(\Zalgebra{U_1}\) and any sequence \(\{\bm\lambda = \bm\lambda(k)\}\) of dominant integral elements such that \(\min\lambda_i(k)\) goes to infinity, we have that \[\lim_{k \to \infty} \rk^{U_1}_{\Wmodk{\bbq_p}}(x) = 1\,.\]
Indeed, fixing a non-principal ultrafilter \(\omega\) on \(\bbn\), the condition above implies that the natural map \[\Zalgebra{U_1} \to \prod_{\omega} \End_{\bbq_p}(\Wmodk{\bbq_p})\] extends to a homomorphism from  its Ore division ring of fractions \(\mathcal{Q}_{U_1}\) as the image of \(x\) is invertible almost everywhere. Hence, since both \(\lim_{\omega} \rk^{U_1}_{\Wmodk{\bbq_p}}\) and \(\rk_{\Qalgebra{U_1}}\) are Sylvester matrix rank functions defined on the division ring \(\mathcal{Q}_{U_1}\), they must coincide for all matrices over \(\mathcal{Q}_{U_1}\) (a division ring has only one Sylvester matrix rank function, since one can always put a matrix in row-echelon form). In particular, they coincide over \(\Qalgebra{U_1}\).
\end{rmk}

The connection between the quantitative statement of Conjecture~\ref{lie-luck-conj-quant} and the \(p\)-adic seting requires passing down to the finite index subgroup \(\Gamma'\) of \(\Gamma\) whose \(p\)-adic closure is contained in the first congruence subgroup. However, for \(\Gamma'\) the results are straightforward. Indeed, if \begin{equation}\label{eq-quantitative-p-adic-analytic}\left|\rk^{U_1}_{\Wmodk{\bbq_p}}(A) - \rk_{\Qalgebra{U_1}}(A)\right| = O\left(\frac{1}{\min\{\lambda_1,\ldots,\lambda_n\}}\right)\end{equation} for any matrix \(A\) over \(\Qalgebra{U_1}\), then it follows directly from Proposition~\ref{prop-transfer-luck-harris}, Lemma~\ref{transfer-Lie-C-p-adic} and Remark~\ref{rmk-twisted-torsion-free} that \(\Gamma'\) satisfies Conjecture~\ref{lie-luck-conj-quant}.

However, once more it is enough to check the validity of~(\ref{eq-quantitative-p-adic-analytic}) only for non-zero elements \(x \in \Zalgebra{U_1}\). This follows from a lemma inspired by \cite[Sec. 6]{fu}:

\begin{lem}\label{lem:approxim-ore-domains} Let \(R\) be an Ore domain with Ore division ring of fractions \(\mathcal{Q}\). Let \(\rk_R\) be the Sylvester matrix rank function on \(R\) induced from \(\mathcal{Q}\) and let \(\rk\) be any other Sylvester matrix rank function on \(R\). For any matrix \(A \in \Mat(R)\), define \(\Delta(A) = |\rk_R A - \rk A|\). Then, there are elements \(x_1,\ldots,x_s \in R\) (depending only on \(A\)) such that \[\Delta(A) \leq \sum_{i=1}^s \Delta(x_i)\,.\]
 \end{lem}
\begin{proof} Let \(\dim_R\) and \(\dim\) be the dimension functions associated to \(\rk_R\) and \(\rk\), respectively. We can also describe \(\Delta(A)\) for \(A\in \Mat_{n\times m}(R)\) using \(\dim_R\) and \(\dim\) as follows:
\[\Delta(A) = \left|\dim_R R^m/R^nA - \dim R^m/R^nA\right|\,.\] Hence, it suffices to prove that, for every finitely presented \(R\)-module \(M\), there are elements \(x_1,\ldots,x_s \in R\) such that \[\left|\dim_R M - \dim M\right|\leq \sum_{i=1}^s \Delta(x_i)\,.\]

Since \(M\) is finitely generated, we can always find \(R\)-linearly independent elements \(m_1,\ldots,m_k \in M\) such that the quotient \(M/(\sum Rm_i)\) is torsion, and so \(\mathcal{Q} \otimes_R M/(\sum Rm_i) = 0\). Indeed, since \(\mathcal{Q} \otimes_R M\) is isomorphic to \(\mathcal{Q}^k\), it suffices to take a \(\mathcal{Q}\)-basis of this space and multiply each element by a suitable non-zero element of \(R\) so that they lie in the image of \(M \to \mathcal{Q}\otimes_R M\). Consider then three short exact sequences:
\begin{align}
    0 \to R^k \to M \to Q \to 0\,,\label{eqn: ses 1}\\
0 \to N \to R^l \to M \to 0\,,\label{eqn: ses 2}\\
0 \to R^j \to N \to P \to 0\,,\label{eqn: ses 3}
\end{align}
where \(N\) is finitely generated and \(P\) and \(Q\) are torsion. Observe that we have that \(k + j = l\) since \(\dim_R\) is additive under short exact sequences.

Note that, since \(\mathcal{Q}\) is flat over \(R\), submodules and quotients of torsion modules are again torsion. Since both \(Q\) and \(P\) are finitely generated, we can find filtrations \[Q = Q_0 > Q_1 > \cdots > Q_e = 0\,,\]
\[P = P_0 > P_1 > \cdots > P_f = 0,\]
such that all of the quotients \(Q_{i-1}/Q_{i}\) and \(P_{i-1}/P_{i}\) are cyclic and torsion. If \(y_1,\ldots,y_e\in R\) and \(z_1,\ldots,z_f \in R\) are non-zero elements that annihilate the quotients \(Q_i/Q_{i+1}\) and \(P_i/P_{i+1}\), respectively, we obtain bounds
\[\dim Q \leq \sum_{i=1}^e \dim R/Ry_i = \sum_{i=1}^e \Delta(y_i)\text{ and } \dim P \leq \sum_{i=1}^f \dim R/Rz_i = \sum_{i=1}^e \Delta(z_i)\,.\]

Let \(s = e + f\) and \(\{x_1,\ldots,x_s\} = \{y_1,\ldots,y_e,z_1,\ldots,z_f\}\). Taking \(\dim\) with respect to \eqref{eqn: ses 1} and \eqref{eqn: ses 3}, respectively, we get that
\begin{equation}\label{eq:1dim}\dim M \leq k + \dim Q \leq \dim_R M + \sum_{i=1}^s \Delta(x_i)\,,\end{equation}
\begin{equation}\label{eq:3dim}\dim N \leq j + \dim P \leq \dim_R N + \sum_{i=1}^s \Delta(x_i)\,.\end{equation}
Now, we take \(\dim\) with respect to \eqref{eqn: ses 2} to obtain that
\[
    l=\dim_R M + \dim_R N \leq \dim N + \dim M\,,
\] which we reorganize as
\begin{equation}\label{eq:2dim}
     \dim_R N - \dim N \leq \dim M - \dim_R M\,.
\end{equation}
Now, plugging inequality~(\ref{eq:3dim}) inside inequality~(\ref{eq:2dim}) and combining with~(\ref{eq:1dim}), we obtain that
\[
    \sum_{i=1}^s -\Delta(x_i) \leq \dim M  - \dim_R M \leq \sum_{i=1}^s \Delta(x_i)\,,
\] which is the desired bound.
\end{proof}

Combining this observation with Theorem~\ref{thm-reduction-padic}, we have managed to prove the following:

\begin{cor}\label{cor-reduction-to-adic} Let \(\schG\) be a simply-connected and semi-simple algebraic group. Assume that for all but finitely many primes \(p\) and every non-zero element \(x \in \Zalgebra{U_1}\), we have that \[\left|\rk^{U_1}_{\Wmodk{\bbq_p}}(x) -1\right| = O\left(\frac{1}{\min\{\lambda_1,\ldots,\lambda_n\}}\right)\,.\] Then, for any finitely generated subgroup \(\Gamma\) of \(\schG\), Conjecture~\ref{lie-luck-conj} holds and Conjecture~\ref{lie-luck-conj-quant} virtually holds.
\end{cor}

\section{The case of \texorpdfstring{\(\scheme{SL}_2\)}{SL₂}}\label{sec-SL2}

In the last section, we have shown in Corollary~\ref{cor-reduction-to-adic} how Conjectures~\ref{lie-luck-conj} and~\ref{lie-luck-conj-quant} for \(\schG\) can be (at least partially) obtained from a corresponding result on completed group algebras of the first congruence subgroup of \(\schG(\bbz_p)\). We now specialize to the case of (\(\scheme{P}\))\(\SL_2\).

Suppose now that all the irreducible factors of the root system \(R\) of \(\schG\) are of type \(A_1\), which is to say that \(\schG\) is a quotient of \(\prod_{i=1}^n \scheme{SL}_2\) by a finite central subgroup. Hence, to establish  Conjecture~\ref{lie-luck-conj} in this case it suffices to do so for \(\schG=\prod_{i=1}^n \scheme{SL}_2\) by the observations in Remark~\ref{rmk-simply-connected}.

Consider the first congruence subgroup 
\[U_1 = \ker\left(\prod_{i=1}^n \scheme{SL}_2(\bbz_p) \to \prod_{i=1}^n \scheme{SL}_2(\bbf_p)\right)\]
of $\schG$. 
By Corollary~\ref{cor-reduction-to-adic}, we only need to show that \[\left|\rk^{U_1}_{\Wmodk{\bbq_p}}(x) - 1\right| = O\left(\frac{1}{\min\{\lambda_1,\ldots,\lambda_n\}}\right)\]
for almost all primes \(p\). This is implied by a stronger theorem of W. Fu \cite{fu}, whose argument we adapt below in the language of this paper for the convenience of the reader.

\begin{thm}\label{our-theorem-sl2} If \(\Gamma\) is any finitely generated subgroup of \(\schG = \prod_{i=1}^n \scheme{SL}_2\), then Conjecture~\ref{lie-luck-conj} holds and Conjecture~\ref{lie-luck-conj-quant} virtually holds.
\end{thm}
\begin{proof}[Proof of Theorem~\ref{thm-cohomology-sl2}] Let \(p\) be an odd prime and\(\ffg\) be the \(\bbq_p\)-Lie algebra of \(U_1\). By Proposition~\ref{prop-embedding-enveloping-alg}, we have an embedding \(\Qalgebra{U_1} \subseteq \widehat{U(\ffg)}\) such that the \(U_1\)-module structure of \(\Wmodk{\bbq_p}\) coincides with the one induced through this embedding (Lemma~\ref{lem-same-module}). 

For \(x \in \Qalgebra{U_1}\), let \(N\) be the \(\Qalgebra{U_1}\)-module \(\Qalgebra{U_1}/x\Qalgebra{U_1}\) and \(M\) be the \(\widehat{U(\ffg)}\)-module \(\widehat{U(\ffg)}/x\widehat{U(\ffg)}\). It is immediately checked that \(M \simeq N \otimes_{\Qalgebra{U_1}}\widehat{U(\ffg)}\) as right \(\widehat{U(\ffg)}\)-modules.

In the language of \cite{fu}, \(x\) is a  \emph{generic element} of \(\widehat{U(\ffg)}\) (cf. \cite[Thm. 5.1]{fu}, \cite[Thm. 4.6 and Thm. 5.4]{ardakov_verma_2014}). Hence, \cite[Thm 3.5]{fu} tells us that there exists a polynomial \(f_M(x_1,\ldots,x_n)\) of total degree \(n-1\) and of individual degree at most \(1\) in each variable such that \[\dim_{\bbq_p} \Hom_{\widehat{U(\ffg)}}(M, \Wmodk{\bbq_p}) \leq f_M(\lambda_1,\ldots,\lambda_n)\] for every \(\bm\lambda = (\lambda_1,\ldots,\lambda_n)\).

Observe that, since the modules \(\Wmodk{\bbq_p}\) are self-dual, i.e. \((\Wmodk{\bbq_p})^* = \Hom_{\bbq_p}(\Wmodk{\bbq_p},\bbq_p)\) is isomorphic to  \(\Wmodk{\bbq_p}\) as a left module, we have that
\begin{align*}
    \dim_{\bbq_p} \Hom_{\widehat{U(\ffg)}}(M, \Wmodk{\bbq_p}) &= \dim_{\bbq_p} \Hom_{\Qalgebra{U_1}}(N, \Wmodk{\bbq_p})\\
    &= \dim_{\bbq_p} \Hom_{\Qalgebra{U_1}}(N \otimes_{\bbq_p} \Wmodk{\bbq_p}, \bbq_p)\\
    &= \dim_{\bbq_p} N \otimes_{\Qalgebra{U_1}} \Wmodk{\bbq_p}\,.
\end{align*}

Hence, we deduce that 
\begin{align*}1 - \rk_{\Wmodk{\bbq_p}}^{U_1}(x) &= \frac{\dim_{\bbq_p}  N \otimes_{\Qalgebra{U_1}} \Wmodk{\bbq_p}}{\dim_{\bbq_p} \Wmodk{\bbq_p}}\nonumber\\
&= \frac{\dim_{\bbq_p} \Hom_{\widehat{U(\ffg)}}(M, \Wmodk{\bbq_p})}{\dim_{\bbq_p} \Wmodk{\bbq_p}} \nonumber\\
&\leq \frac{f_M(\lambda_1,\ldots,\lambda_n)}{\prod_{i=1}^n (\lambda_i + 1)} = O\left(\frac{1}{\min\{\lambda_1,\ldots,\lambda_n\}}\right)\,. \label{eqn-ineq-poly}\end{align*}
The result now follows from Corollary~\ref{cor-reduction-to-adic}
\end{proof}

\section{3-manifolds}\label{sec-manifolds}

Let \(M\) be an orientable hyperbolic \(3\)-manifold, i.e. a \(3\)-manifold admitting a complete Riemannian metric of constant negative curvature. We assume that its fundamental group \(\pi_1(M)\) is finitely generated. By a theorem of P. Scott (\cite{Scott1973}), under this assumption \(\pi_1(M)\) is of type \(F\) over \(\bbc\).

Marden's tameness theorem, proven independently by I. Agol in \cite{agol_tameness_2004} and D. Calegari and D. Gabai in \cite{calegari_shrinkwrapping_2006}, says that \(M\) is homeomorphic to the interior of a compact \(3\)-manifold \(\overline{M}\). In particular, \(M\) can be decomposed into a compact core \(C\) that is homotopy equivalent to \(M\) (in fact, \(C\) is homeomorphic to \(\overline{M}\)) and a union of finitely many ends \(E_i\), each one homeomorphic to a product \(S_i \times (0,\infty)\), where \(S_i\) is a closed surface (see \cite[Sec. 5]{ji_mardens_2008}). Moreover, if \(M\) is of finite volume, then all the \(S_i\) are tori, which follows from the Thick-Thin decomposition of \(M\) \cite[Cor. 4.2.18]{martelli_introduction_2023}. Observe that by the Poincaré duality of manifolds with boundary, if each \(S_i\) has genus \(g_i\), then the Euler characteristic of \(M\) is \(\chi(M) = \sum (1-g_i)\).

We can identify \(\Gamma = \pi_1(M)\) with a torsion-free subgroup of the group \(\operatorname{Isom}^+(\mathbb{H}^3)\) of orientation-preserving isometries of the hyperbolic \(3\)-space \(\mathbb{H}^3\). This group, in turn, is isomorphic to \(\scheme{PSL}_2(\bbc) = \scheme{SL}_2(\bbc)/\{\pm\operatorname{Id}\}\), which we know satisfies Conjecture~\ref{conjecture-qualitative} by Theorem~\ref{thm-cohomology-sl2}. Identifying \(\bm\lambda = \lambda \in \bbn\), we note that the representations \(\Wmodk{} = \Sym^\lambda \bbc^2\) of \(\scheme{SL}_2(\bbc)\) that induce representations of \(\scheme{PSL}_2(\bbc)\) are precisely those for which \(-\operatorname{Id}\) acts trivially, i.e. the representations where \(\lambda\) is even.

In \cite{menal-ferrer_twisted_2012}, J. Porti and P. Menal-Ferrer calculated the dimensions of the homology groups \(\HH_i(M, \Sym^\lambda \bbc^2) = \HH_i(\Gamma, \Sym^\lambda \bbc^2)\).

\begin{prop}[{cf. \cite[Cor. 3.6]{menal-ferrer_twisted_2012}}]\label{prop-porti-ferrer} Let \(M\) a hyperbolic \(3\)-manifold with finitely generated fundamental group and whose ends contain exactly \(k\) toroidal cusps. For \(\lambda\in\bbn\) even, we have that 
\[\dim_\bbc \HH_i(M, \Sym^{\lambda}(\bbc^2)) = \begin{cases}
    0,&\text{if } i \in \{0,3\}\,,\\
    k,&\text{if }i = 2,\\
    k - (\lambda + 1)\chi(M),&\text{if } i = 1\,.
\end{cases}\]
\end{prop}

\begin{rmk} There are lifts of \(\Gamma\) to subgroups of \(\scheme{SL}_2(\bbc)\) and, in \cite{menal-ferrer_twisted_2012}, Porti and Menal-Ferrer obtained formulas for the dimension of \(\HH_i(M, \Sym^{\lambda}(\bbc^2))\) for all \(\lambda\), which for odd \(\lambda\) depends on how the \(\bbz^2\) subgroups attached to the toroidal cusps lift to \(\scheme{SL}_2(\bbc)\). Moreover, in the case where \(M\) is closed, the vanishing of the cohomology groups \(\HH_i(M,\Sym^\lambda(\bbc^2))\) follows from a theorem by M. Raghunathan \cite{raghunathan_first_1965}.
\end{rmk}

The previous result together with Proposition~\ref{prop-homology-approx} yields the following.

\begin{cor}\label{cor-bettino} For \(M\) as in Proposition~\ref{prop-porti-ferrer}, we have that
\[b_i^{(2)}(M) = \begin{cases}0,&\text{if }i \neq 1\,,\\-\chi(M),&\text{if }i = 1\,.\end{cases}\]
\end{cor}

In \cite{luckL2TopologicalInvariants3manifolds1995}, J. Lott and W. Lück gave a general formula for the \(\ell^2\)-Betti numbers of every compact \(3\)-manifold \(M\). Using Marden's tameness theorem to reduce the finitely generated to the compact case, one recovers the formula in Corollary~\ref{cor-bettino}. By repeatedly applying the \(\ell^2\)-Mayer-Vietoris sequence, the Prime decomposition theorem reduces the general case to that of an irreducible \(M\) with infinite \(\pi_1(M)\) and the Loop theorem reduces to the case where the boundary of \(M\) is incompressible. Furthermore, by doubling \(M\) we may assume that it is closed. Then, the Geometrization theorem reduces the analysis even further to the cases where \(M\) is either Seifert fibered or hyperbolic.

Both cases are where the core of the calculation lies. That the \(\ell^2\)-Betti numbers of Seifert fibered \(3\)-manifolds vanish can be seen, for instance, as a consequence of the existence of a virtual \(S^1\)-action \cite[Thm. 1.40]{luck_l2-invariants_2002} or of the virtual algebraic fibering of its fundamental group \cite[Thm. 7.2(5)]{luck_l2-invariants_2002}. The vanishing of \(\ell^2\)-Betti numbers for closed hyperbolic \(3\)-manifolds is implied by a theorem of J. Dodziuk \cite{dodziuk_l2_1979}, and thus the trickiest part of Lott and Lück's argument is the deduction of the finite volume hyperbolic case from the closed one (for more details, see \cite[Sec. 4.2]{luck_l2-invariants_2002}). This case is covered by Corollary~\ref{cor-bettino}, and also follows from their virtual algebraic fibering. We also point out that under Corollary~\ref{cor-bettino} the \(\ell^2\)-Betti numbers of \(M\) are deduced directly from the representation-theoretic aspects of hyperbolic geometry, avoiding the use of the JSJ-decomposition.

\section{Cusp forms for \texorpdfstring{\(\SL_2\)}{SL₂}}\label{sec-cusp-forms}

The quantitative statement of Conjecture~\ref{conjecture-quantitative} has applications in the theory of automorphic forms for \(\SL_2\). We will give here a brief account of the theory of cusp forms necessary for stating our results, and we direct the reader to \cite{hidaPordinaryCohomologyGroups1993, borelAutomorphicFormsRepresentations1979} for details and further references.

Let \(F\) be a number field with \(r_1 + 2r_2 = (F\colon \bbq)\) real and complex embeddings respectively, and let \(I_\bbr = \{\tau_1,\ldots,\tau_{r_1}\}\) and \(I_\bbc = \{\sigma_1,\ldots,\sigma_{r_2}\}\) be complete representative sets of real and complex embeddings, respectively, corresponding to the archimedean places of \(F\). We define \(I_\infty = I_\bbr \cup I_\bbc\) and \(I = I_\infty \cup \overline{I_\bbc}\). Consider the group \[G = \SL_2(F \otimes_\bbq \bbr) = \prod_{r_1} \SL_2(\bbr) \times \prod_{r_2} \SL_2(\bbc)\,.\] Letting \(\mathcal{O}_F\) denote the ring of integers of \(F\), we fix a congruence subgroup \(\Gamma\) of \(\SL_2(\mathcal{O}_F)\) that will be the \emph{level} of our automorphic forms. Then, we get an embedding \(\Gamma \to G\) by taking the product of all maps induced by the \(\tau_i\) and the \(\sigma_j\).

By quotienting out on the right the product \(\prod_{r_1} \scheme{SO}_2(\bbr) \times \prod_{r_2} \scheme{SU}_2(\bbc)\) of the standard maximal compact subgroups of \(G\), we obtain the locally symmetric space
\[\mathbb{H} = \prod_{r_1} \mathbb{H}^2 \times \prod_{r_2} \mathbb{H}^3\] associated to \(G\) on which \(\Gamma\) acts. The quotient \(M = \Gamma\backslash \mathbb{H}\) is a \((2r_1 + 3r_2)\)-dimensional orbifold of finite volume, and it is a manifold if \(\Gamma\) is torsion-free. In this case, the Borel-Serre compactification attaches a boundary \(\partial M\) to the manifold \(M = \Gamma\backslash \mathbb{H}\), and the resulting compact manifold has the same fundamental group \(\Gamma\).

Denote by \(\mathfrak{G} = \prod_{r_1} \fsl_2(\bbr) \times \prod_{r_2} \fsl_2(\bbc)\) the real Lie algebra of \(G\) and by \(\mathfrak{G}_\bbc = \prod_{r_1 + 2r_2} \fsl_2(\bbc)\) its complexification. The center of the universal enveloping algebra \(U(\mathfrak{G}_\bbc)\) is a polynomial algebra in \(r_1 + 2r_2\)-generators called the \emph{Casimir operators}. For each embedding \(\delta \in I\), we obtain one Casimir operator \(\Delta_{\delta}\) given by the image under \(\delta\) of \(\frac{1}{4}(EF + FE + H^2/2)\), where \[E = \begin{pmatrix}0 & 1 \\ 0 & 0\end{pmatrix}, \quad F = \begin{pmatrix}0 & 0\\1 & 0\end{pmatrix}, \quad H = \begin{pmatrix}1 & 0\\0 & -1\end{pmatrix}\] form the standard basis of \(\fsl_2(\bbc)\).

Let \(\bm{k} = (k_{\delta})_{\delta \in I_\infty}\) be a \((r_1 + r_2)\)-tuple of non-negative integers indexed by the archimedean places of \(F\). Let \(V_{\bm{k}} = \bigotimes_{\delta \in I_\infty} V_\delta\) be the complex vector space given by the tensor product of \(V_\delta = \bbc\) for all the real embeddings \(\delta\) and \(V_\delta=\operatorname{Sym}^{2k_{\delta}-2} \bbc\) for all the complex embeddings \(\delta\). It naturally has the structure of a right module over \(\prod_{r_1} \bbc^\times \times \prod_{r_2} \GL_2(\bbc)\). We denote by \(S_{\bm k}(\Gamma)\) the finite-dimensional complex vector space of cusp forms of weight \(\bm k\) and level \(\Gamma\).

The current literature shows great interest in the asymptotic behaviour of the dimensions \(\dim_\bbc S_{\bm{k}}(\Gamma)\), either for a fixed weight \(\bm{k}\) and a residual chain of levels \(\Gamma_i\) (\cite{savinLimitMultiplicitiesCusp1989, sarnakBoundsMultiplicitiesAutomorphic1991,calegariBoundsMultiplicitiesUnitary2009,finisCohomologyLatticesSL2010}), or for a fixed level \(\Gamma\) and increasing sequence of weights \(\{\bm{k}_i\}\) (\cite{shimizuDiscontinuousGroupsOperating1963,finisCohomologyLatticesSL2010,marshallBoundsMultiplicitiesCohomological2012,fu}). 

In this work, we are concerned with the latter case of growing weights for a fixed level. When \(F\) is totally real, the precise growth rate of \(S_{\bm{k}}(\Gamma)\) was obtained by H. Shimizu in \cite{shimizuDiscontinuousGroupsOperating1963} as
\[\dim_{\bbc} S_{\bm{k}}(\Gamma) \sim \prod_{\tau \in I} k_{\tau}\,.\] When \(F\) is not totally real, the question of what the precise growth rate of \(S_{\bm{k}}(\Gamma)\) should be is very much open. Currently, the best general upper bound is the one given in \cite{fu}:

\begin{thm}[{\cite[Thm 1.2]{fu}}]\label{thm-fu-cusp-forms} If \(F\) is not totally real, then \[\dim_{\bbc} S_{\bm{k}}(\Gamma) = O\left(\frac{\prod_{\tau \in I_\bbr} k_{\tau} \times \prod_{\sigma \in I_\bbc}k_{\sigma}^2}{\min\{k_{\delta}\colon \delta \in I_\infty\}}\right)\,.\]
\end{thm}

If \(F\) is imaginary quadratic, i.e. when \(\Gamma\) is a congruence subgroup of a Bianchi group, this result combined with the lower bound from \cite{finisCohomologyLatticesSL2010} gives the precise growth rate.

\begin{cor}[{\cite[Cor 1.3]{fu}}] If \(F\) is imaginary quadratic, then \(\bm{k} = (k)\) and \[\dim_{\bbc} S_k(\Gamma) \sim k\,.\]
\end{cor}

Assume that \(\Gamma\) is torsion-free and that \(k_\delta \geq 3\) for each real or complex embedding \(\delta\), and let \(\lambda_\delta = k_\delta - 2\) with the convention that \(k_{\delta} = k_{\overline{\delta}}\). If we set \(\bm\lambda = (\lambda_\delta)\), then \(\Wmodk{}\) becomes a \(\Gamma\)-module through the embedding \(\Gamma \to \prod_{r_1 + 2r_2} \SL_2(\bbc)\) given by the product of all the embeddings \(\Gamma \to \SL_2(\bbc)\) induced by the \(\delta \in I\).

It is well known that the spaces \(S_{\bm{k}}(\Gamma)\) can be computed through the group cohomology of \(\Gamma\), a result that goes by the names of the Eichler-Shimura isomorphism, the Matsushima-Murakami formula and, in the present formulation, is due to G. Harder. There is a natural procedure to attach to each cusp form \(f \in S_{\bm{k}}(\Gamma)\) a \(\Gamma\)-invariant \(\Wmodk{}\)-valued closed differential \((r_1 + 2r_2)\)-form \(\omega(f)\) on \(\mathbb{H}\), thus inducing a map \(\omega\colon S_{\bm{k}}(\Gamma) \to \HH^{r_1+2r_2}(\Gamma,\Wmodk{})\). The reader may consult \cite[Sec. 2.4]{hidaCriticalValuesLfunctions1994} for the details. 
The precise statement of the Eichler-Shimura isomorphism that we need is as follows:

\begin{thm}[{\cite[Thm. 1.1]{hidaPordinaryCohomologyGroups1993}, \cite[Cor.2.2]{hidaCriticalValuesLfunctions1994}}]\label{thm-eichler-shimura} The map \(\omega\) yields an injection \[S_{\bm{k}}(\Gamma)\to \HH_{\operatorname{sq}}^{r_1+2r_2}(\Gamma,\Wmodk{})\,,\] where \(\HH_{\operatorname{sq}}^{r_1+2r_2}(\Gamma,\Wmodk{})\) denotes the subspace of ordinary cohomology represented by square-integrable forms.
\end{thm}

\begin{rmk} We will also use that \(\omega(S_{\bm k}(\Gamma))\) is contained in the subspace of compactly supported cohomology \(\HH_{\operatorname{c}}^{r_1+2r_2}(\Gamma,\Wmodk{})\), which follows from \cite[Cor. 5.5]{borelStableRealCohomology1981}.
\end{rmk}


The proof of Theorem~\ref{thm-fu-cusp-forms} applies \cite[Thm. 3.5]{fu}, which is the core result used in the proof of Theorem~\ref{thm-cohomology-sl2}, to the completed homology of \(\Gamma\backslash \mathbb{H}\) as constructed by M. Emerton in \cite{emertonInterpolationSystemsEigenvalues2006} and combines this with a spectral sequence in the profinite homology of \(U_1\) to obtain the desired upper bound for \(\HH_*(\Gamma, \Wmodk{})\). We propose a different proof of this result that avoids the use of completed cohomology entirely and depends solely on the Eichler-Shimura isomorphism from Theorem~\ref{thm-eichler-shimura} and the results of this paper.

\begin{proof}[Proof of Theorem~\ref{thm-fu-cusp-forms}] Since we have that \(\lambda_\delta = k_\delta - 2\), for \(\bm\lambda = (\lambda_1,\ldots,\lambda_{r_1+2r_2})\) it suffices to show that \[\dim_\bbc \HH_{r_1+r_2}(\Gamma, \Wmodk{}) = O\left(\frac{\prod_{i=1}^{r_1+2r_2} \lambda_i}{\min \lambda_i} \right)\,.\]
Due to the fact that \(\Gamma\) is the fundamental group of a compact manifold by the Borel-Serre construction, we conclude that \(\Gamma\) is of type \(FP_\infty\) over \(\bbc\). 

In order to apply Proposition~\ref{prop-0-betti-number}, we must first show that \(\beta_{r_1+r_2}^{(2)}(\Gamma) = 0\). If \(r_1 = 0\) and \(r_2 = 1\), i.e. if \(F\) is imaginary quadratic, then \(M\) is a hyperbolic \(3\)-manifold of finite volume and we would be done by invoking Corollary~\ref{cor-bettino}. However, in the general non-totally real case, we need something more.

By the proportionality principle of \(\ell^2\)-Betti numbers (\cite[Cor. 4.2.11]{lohErgodicTheoreticMethods}, \cite[Thm. 3.183, Cor. 7.37]{luck_l2-invariants_2002}), the \((r_1+r_2)\)-th \(\ell^2\)-Betti number of \(\Gamma\) and the \((r_1+r_2)\)-th \(\ell^2\)-Betti number of \(\Lambda\) are proportional to each other for any other lattice \(\Lambda\) in \(G\). Since \(\mathbb{H}\) is a symmetric space of non-compact type and dimension \(2r_1 + 3r_2\), the \((r_1+r_2)\)-th \(\ell^2\)-Betti number of \(\Lambda\backslash\mathbb{H}\) vanishes for any uniform lattice \(\Lambda\) assuming that \(r_2 > 0\) (see \cite[Thm. 5.12]{luck_l2-invariants_2002}), and hence so does \(b_{r_1+r_2}^{(2)}(\Gamma)\). The result now follows from Proposition~\ref{prop-0-betti-number} and the duality between the compactly supported cohomology group \(\HH^{r_1+2r_2}_{\operatorname{c}}(\Gamma,\Wmodk{})\) and the homology group \(\HH_{r_1+r_2}(\Gamma,\Wmodk{})\).
\end{proof}

\section{Characters and induced von Neumann ranks}\label{sec-vonNeumann}

If \(\Gamma\) is a countable group, let \(\ell^2(\Gamma)\) be the complex Hilbert space on an orthonormal basis given by the elements of \(\Gamma\). Hence, \(\Gamma\) acts on \(\ell^2(\Gamma)\) on both sides as unitary operators. Any \(r\times s\) matrix \(A\) over \(\bbc[\Gamma]\) induces a bounded linear operator \(\phi^A_{\Gamma}\colon \ell^2(\Gamma)^s \to\ell^2(\Gamma)^r\) between Hilbert spaces by left multiplication. The \emph{von Neumann rank} of \(A\) is defined as \[\rk_{\Gamma}A =\dim_{\Gamma}\overline{\Image\phi^A_{\Gamma}}=s-\dim_{\Gamma}\Ker \phi^A_{\Gamma}\,,\] where \(\dim_\Gamma\) is the \emph{von Neumann dimension} of \(\Gamma\) (see \cite[Def. 1.10]{luck_l2-invariants_2002}). If \(\Gamma\) is not countable, then for every matrix \(A\) over \(\bbc[\Gamma]\) we define \(\rk_\Gamma A\) to be \(\rk_{\Gamma'} A\), where \(\Gamma'\) is any finitely generated subgroup of \(\Gamma\) containing the elements that appear in the entries of \(A\) (this rank does not depend on \(\Gamma'\), see \cite[Thm. 6.29]{luck_l2-invariants_2002}).

We recall that if \(\Gamma\) is finite, then \(\ell^2(\Gamma) = \bbc[\Gamma]\), a Hilbert \(\Gamma\)-module \(V\) is simply a \(\bbc[\Gamma]\)-module and we have that \[\dim_{\Gamma} V = \frac{\dim_{\bbc} V}{|\Gamma|}\quad\text{and}\quad\rk_\Gamma A = \frac{\rk_\bbc A}{|\Gamma|}\,.\]

The following result is a particular case of the \emph{sofic Lück approximation theorem}.

\begin{thm}[{\cite{luck_approximation_1994,jaikin-zapirain_base_2019}}]\label{sofic-luck} Let \(\Gamma_1 \geq \Gamma_2 \geq \cdots\) be a chain of normal subgroups of \(\Gamma\) of finite index with \(\bigcap \Gamma_i = \{1\}\). For each matrix \(A\) over \(\bbc[\Gamma]\), let \(A_i\) denote its image in \(\Mat(\bbc[\Gamma/\Gamma_i])\). We have that \[\lim_{i \to \infty} \rk_{\Gamma/\Gamma_i} A_i = \lim_{i \to \infty} \frac{\rk_\bbc A_i}{|\Gamma: \Gamma_i|} = \rk_{\Gamma} A\,.\]
\end{thm}

There is a natural generalization of the von Neumann rank of finite groups to arbitrary fields. Suppose \(Q\) is a finite group. Then, for any field \(L\) and any matrix \(A \in \Mat(L[Q])\), we define \[\rk_Q^L A = \frac{\rk_L A}{|Q|}\,.\] Observe that \(\rk_Q^\bbc\) is simply the von Neumann rank \(\rk_Q\). For any field extension \(L \leq K\), we have that \(\rk_Q^L A = \rk_Q^K A\) for every matrix \(A \in \Mat(L[Q])\). If \(R\) is an integral domain with field of fractions \(K\) and \(A \in \Mat(R[Q])\), we define \(\rk_Q^R A = \rk_Q^K A\).


We shall now describe S. Kionke's construction of the \(\phi\)-rank on \(\bbc[\Gamma]\) induced from a character \(\chi\colon Z \to \bbc\) on a finite subgroup \(Z \leq \Gamma\). Although his construction applies to arbitrary finite subgroups and characters, we shall restrict ourselves to the case where \(Z\) is contained in the center \(Z(\Gamma)\) of \(\Gamma\). In particular, our characters \(\chi\) come from representations \(Z \to \bbc^\times\), which are all irreducible. Following \cite[Def. 4.10]{Kionke2018}, the induced character \(\phi=\operatorname{ind}^\Gamma_Z \chi\colon \Gamma \to \bbc\) is defined by:
\[\phi (g) = \begin{cases}
    \chi(g)\text{ , if }g \in Z\,,\\
    0\text{ otherwise.}
\end{cases}\]
There is a bijection between characters \(\phi\) of \(\Gamma\) and Hilbert \(\Gamma\)-modules \(\mathcal{H}_\phi\) (\cite[Prop. 2.6]{Kionke2018}). In particular, for \(\phi = \operatorname{ind}^\Gamma_Z \chi\) there exists a Hilbert space \(\mathcal{H}_\phi\) endowed with a \(\bbc[\Gamma]\)-\(\bbc[\Gamma]\)-bimodule structure such that each \(g \in \Gamma\) acts on both sides by isometries.

The commutant of the right \(\Gamma\)-action on the space of bounded operators of \(\mathcal{H}_\phi\) is a finite von Neumann algebra \(\mathcal{N}_\phi\) and the left action gives rise to an algebra homomorphism \(\bbc[\Gamma] \to \mathcal{N}_\phi\). There exists a special element \(e \in \mathcal{H}_\phi\) of norm one such that \(ge = eg\) for all \(g \in \Gamma\) and the translates \(\Gamma e\) span a dense subspace of \(\mathcal{H}_\phi\), and this von Neumann algebra \(\mathcal{N}_\phi\) comes endowed with a faithful trace \(\operatorname{tr}_\phi\) given by \[\operatorname{tr}_\phi(a) = \langle a\cdot e, e\rangle\,.\] For \(g \in G\), we also recover the character \(\phi(g)\) as the trace \(\operatorname{tr}_\phi(g)\).  These properties describe uniquely the Hilbert \(\Gamma\)-module \(\mathcal{H}_\phi\) up a \(\Gamma\)-equivariant isometric isomorphism. For our induced character \(\phi\), we shall now describe the Hilbert \(\Gamma\)-module \(\mathcal{H}_\phi\) and the von Neumann algebra \(\mathcal{N}_\phi\) explicitly.

Let \[e_\chi = \frac{1}{|Z|} \sum_{h \in Z} \chi(h^{-1})h\] be the primitive idempotent \(\bbc[Z]\) corresponding to \(\chi\). The corner algebra \(\bbc[Z]e_\chi\) is one dimensional over \(\bbc\), and \(\bbc[Z]\) acts on it through the character \(\chi\). Since \(Z\) is central in \(\Gamma\), \(e_\chi\) is also a central idempotent in \(\bbc[\Gamma]\) and in its group von Neumann algebra \(\mathcal{N}(\Gamma)\), the commutant of the right \(\Gamma\)-action on \(\ell^2(\Gamma)\). If we take \(\mathcal{H} = \ell^2(\Gamma)e_\chi\), then we see it is a closed subspace of \(\ell^2(\Gamma)\) invariant under the \(\Gamma\)-action on both sides. By the centrality of \(e_\chi\), we have \(ge_\chi = e_\chi g\) for all \(g \in \Gamma\), and moreover 
\begin{align*}
    \langle e_\chi, e_\chi\rangle &= \frac{1}{|Z|^2} \sum_{g,h \in Z} \chi(g^{-1})\overline{\chi(h^{-1})}\langle g,h\rangle\\
    &= \frac{1}{|Z|^2} \sum_{g \in Z} \chi(g)^{-1}\chi(g) = \frac{1}{|Z|}\,.
\end{align*}
Hence, by re-scaling the restriction of the inner product of \(\ell^2(\Gamma)\) to \(\mathcal{H}\) by a factor of \(|Z|\), we obtain an inner product \(\langle \cdot, \cdot \rangle_\phi\) on \(\mathcal{H}\) such that \(\langle e_\chi, e_\chi \rangle_\phi = 1\). Since \(\mathcal{H}\) is the image of the ``right multiplication by \(e_\chi\)'' map \(\ell^2(\Gamma) \to \ell^2(\Gamma)\) and the elements of \(\Gamma\) span a dense subspace of \(\ell^2(\Gamma)\), the \(\Gamma\)-translates of \(e_\chi\) span a dense subspace of \(\mathcal{H}\). To see that \(\mathcal{H} \simeq \mathcal{H}_\phi\), it only remains to show that \[\phi(g) = \langle g\cdot e_\chi, e_\chi \rangle_\phi\,.\] However:
\begin{align*}
    \langle g\cdot e_\chi, e_\chi \rangle_\phi &= |Z| \langle g\cdot  e_\chi, e_\chi\rangle \\
    &= \frac{1}{|Z|} \sum_{h_1 \in Z} \sum_{h_2 \in Z} \chi(h_1^{-1}h_2) \langle gh_1,h_2\rangle\\
    &=\begin{cases}
        0\text{, if }g \not\in Z\,,\\
        \frac{1}{|Z|} \sum_{h \in Z} \chi(g) = \chi(g)\text{, if }g \in Z\,,
    \end{cases}
\end{align*}
which is precisely \(\phi(g)\).

We have a decomposition \[\ell^2(\Gamma) \simeq \mathcal{H} \oplus \ell^2(\Gamma)(1-e_\chi)\] of Hilbert \(\Gamma\)-modules and a corresponding decomposition \[\mathcal{N}(\Gamma) = \mathcal{N}(\Gamma)e_\chi \oplus \mathcal{N}(\Gamma)(1-e_{\chi})\] of its group von Neumann algebra. This allows us to identify \(\mathcal{N}_\phi\) with \(\mathcal{N}(\Gamma)e_{\chi}\).

The faithful trace \(\operatorname{tr}_\phi\) on \(\mathcal{N}_\phi\) is given by \[\operatorname{tr}_\phi(A) = \langle A\cdot e_\chi, e_\chi \rangle_\phi = |Z| \langle A \cdot e_\chi, e_\chi\rangle\,.\] Any closed subspace \(M \leq \mathcal{H}_\phi^n\) that is invariant under the right \(\Gamma\)-action comes endowed with a projection \(P = (P_{ij}) \in \mathrm{Mat}_n(\mathcal{N}_\phi)\) which is the unique self-adjoint idempotent element satisfying \(P\mathcal{H}_\phi^n = M\). We define the von Neumann dimension of \(M\) as the trace of \(P\):
\[\dim_\phi M = \sum_{i=1}^n\operatorname{tr}_\phi(P_{ii}) = \sum_{i=1}^n\langle P_{ii}\cdot e_\chi, e_\chi \rangle_\phi\,.\] If \(A\) is an \(n\times m\) matrix over \(\mathcal{N}_\phi\), we define its \(\phi\)-rank and its \(\phi\)-nullity as: \[\rk_\phi(A) = \dim_\phi(\overline{\operatorname{im}(A)})\,,\quad\mathrm{null}_\phi(A) = \dim_\phi(\ker(A))\] where we think of ``left multiplication by \(A\)'' as a map \(\mathcal{H}_\phi^n \to \mathcal{H}_\phi^m\) of right \(\bbc[\Gamma]\)-modules. By \cite[Lem. 2.4]{Kionke2018}, we have \[\rk_\phi(A) + \operatorname{null}_\phi(A) = n\,.\]

If \(\Gamma'\) is a subgroup of \(\Gamma\) and we define \(Z' = Z\cap \Gamma\) and \(\chi'\) to be the restriction of \(\chi\) to \(Z\), then we also have a character \(\phi' = \operatorname{ind}^{\Gamma'}_{Z'} \chi'\) given by \[\phi'(g) = \begin{cases}
    \chi'(g)\text{, if }g \in Z'\,,\\
    0\text{, otherwise}.;
\end{cases}\] Observe that \(\phi'\) equals the restriction of \(\phi\) to \(\Gamma'\) as group characters. By the previous remarks, up to a \(\Gamma'\)-equivariant isometry there is a unique Hilbert space \(\mathcal{H}_{\phi'}\) endowed with a \(\bbc[\Gamma']-\bbc[\Gamma']\)-bimodule structure with a unitary vector whose \(\Gamma\)-translates generate a dense subspace and a faithful tracial von Neumann algebra \(\mathcal{N}_{\phi'}\) given by the commutant of the right \(\bbc[\Gamma']\)-action such that \(\phi'(g) = \operatorname{tr}_{\phi'}(g)\). If \[e_{\chi'} = \frac{1}{|Z'|} \sum_{h \in Z'} \chi'(h^{-1})h\] denotes the central idempotent of \(\bbc[Z']\) associated to \(\chi'\), then \(\mathcal{H}_{\phi'} = \ell^2(\Gamma')e_{\chi'}\) after a normalization of the inner product and \(\mathcal{N}_{\phi'} = \mathcal{N}(\Gamma)e_{\chi'}\).

Identifying \(\ell^2(\Gamma)\) with the Hilbert space completion of \( \ell^2(\Gamma') \otimes_{\bbc[\Gamma']} \bbc[\Gamma]\) as right \(\bbc[\Gamma]\)-modules induces a homomorphism \(i\colon \mathcal{N}(\Gamma') \to \mathcal{N}(\Gamma)\) of von Neumann algebras that maps each \(g \in \Gamma'\) to itself as an element of \(\bbc[\Gamma] \leq \mathcal{N}(\Gamma)\). Since \((1-e_{\chi'})e_\chi = 0\) and \(Ae_{\chi'} = Be_{\chi'}\) if and only if \(A-B \in (1-e_{\chi'})\mathcal{N}(\Gamma')\), this induces a well defined von Neumann algebra homomorphism \(i_{\phi',\phi}\colon \mathcal{N}_{\phi'} \to \mathcal{N}_{\phi}\) by means of \[i_{\phi',\phi}(Ae_{\chi'}) = i(A)e_{\chi}\] for any element \(A \in \mathcal{N}(\Gamma')\).

\begin{lem}\label{lem-independence-of-subgroup} Let \(\Gamma'\), \(Z'\) and \(\chi'\) be as above. Then:
\begin{enumerate}[label=(\alph*),wide]
    \item For any \(A \in \mathcal{N}_{\phi'}\) we have \[\operatorname{tr}_{\phi'}(A) = \operatorname{tr}_{\phi}(i_{\phi',\phi}(A))\,.\]
    \item If \(A\) is an \(n\times m\) matrix over \(\mathcal{N}_{\phi'}\), then \[\rk_{\phi'}(A) = \rk_{\phi}(i_{\phi',\phi}(A))\text{ and }\operatorname{null}_{\phi'}(A) = \operatorname{null}_{\phi}(i_{\phi',\phi}(A))\,.\]
\end{enumerate}
In particular, \(\rk_{\phi'}\) is the pullback of \(\rk_{\phi}\) along \(i_{\phi',\phi}\) as Sylvester matrix rank functions on \(\mathcal{N}_{\phi'}\).
\end{lem}
\begin{proof} (a) Given \(A\) in \(\mathcal{N}(\Gamma')\), the element \(i(A)\) is defined as the unique bounded right \(\Gamma\)-invariant operator on \(\ell^2(\Gamma)\) that acts on the dense subspace \(M = \ell^2(\Gamma') \otimes_{\bbc[\Gamma']} \bbc[\Gamma]\) as left multiplication by \(A\). Let \(T\) be a complete set of representatives of \(Z/Z'\). As an element of \(M \leq \ell^2(\Gamma)\), we have \begin{align*}e_\chi &= \frac{1}{|Z|}\sum_{z \in Z} \chi(z^{-1})z^{-1}\\ &= \frac{1}{|Z\colon Z'|} \sum_{t \in T} \left(\frac{1}{|Z'|} \sum_{h \in Z'} \chi(h^{-1})h\right)\chi(t^{-1})t\\
&= \frac{1}{|Z \colon Z'|} \sum_{t \in T} \chi(t^{-1})( e_{\chi'} t)\,.\end{align*}
Hence: 
\begin{align*}
    \langle i_{\phi',\phi}(Ae_{\chi'})e_{\chi}, e_{\chi}\rangle_{\phi} &= |Z| \langle i(A)e_{\chi}, e_{\chi} \rangle \\
    &= \frac{|Z|}{|Z\colon Z'|^2} \sum_{t_1, t_2 \in T} \chi(t_1^{-1}t_2) \langle Ae_{\chi'} t_1, Ae_{\chi'}  t_2\rangle\\
    &= \frac{|Z|}{|Z\colon Z'|} \sum_{t \in T} \langle Ae_{\chi'}  t, Ae_{\chi'}  t\rangle
\end{align*}
where the last equality comes from \(Ae_{\chi'} t_1\) and \(Ae_{\chi'}  t_2\) being supported on different and thus orthogonal cosets of \(\ell^2(\Gamma')\) unless \(t_1 = t_2\). Since \(\Gamma\) and therefore \(T\) acts on \(\ell^2(\Gamma)\) by isometries, we get:
\begin{align*}
    \langle i_{\phi',\phi}(Ae_{\chi'})e_{\chi}, e_{\chi} \rangle_{\phi} &= \frac{|Z|}{|Z\colon Z'|^2} \sum_{t \in T} \langle Ae_{\chi'}  t, Ae_{\chi'}  t\rangle\\
    &= |Z'| \langle Ae_{\chi'}, Ae_{\chi'}\rangle \\
    &= \operatorname{tr}_{\phi'}(Ae_{\chi'})\,.
\end{align*}

(b) follows immediately from (a), since if \(P\) denotes the orthogonal projection onto the closure of image of \(A\) then \(i_{\phi',\phi}(P)\) is the orthogonal projection onto the closure of the image of \(i_{\phi',\phi}(A)\).
\end{proof}

\begin{rmk}\label{rmk-twisted-torsion-free} If \(Z = \{1\}\), then \(\operatorname{tr}_\phi\) is simply the characteristic function \(\operatorname{tr}_\phi(g) = \delta_{1,g}\) and \(\mathcal{H}_\phi \simeq \ell^2(\Gamma)\). Hence, the von Neumann rank \(\rk_\phi\) coincides with the classical von Neumann rank \(\rk_\Gamma\). Even when \(Z \neq \{1\}\), if \(\Gamma'\) is a subgroup of \(\Gamma\) that intersects \(Z\) trivially and \(A\) is a matrix over \(\bbc[\Gamma']\), we have that \(\rk_\phi A = \rk_{\Gamma'} A\) by Lemma~\ref{lem-independence-of-subgroup}. In particular, this is the case if \(\Gamma'\) is torsion-free.

When \(\Gamma\) has torsion central elements, then distinct central characters lead to distinct rank functions. This is why the hypothesis on the central characters is essential to the statement of Conjecture~\ref{conjecture-qualitative}. As an example, consider \(\Gamma\) the cyclic group of order \(2\) inside \(\SL_2(\bbc)\) generated by \(g = -\operatorname{Id}\) and the element \(g-1\) inside \(\bbc[\Gamma]\). Then, for the trivial character \(\bm{1}(g) = 1\), we have that \(\rk_{\bm{1}}(g-1) = 0\), while for the non-trivial character \(\bm{-1}(g) = -1\) we have that \(\rk_{\bm{-1}}(g-1) = 1\). This reflects on the cohomology of \(\Gamma\) with coefficients in \(\Sym^\lambda \bbc^2\) by the appearance of two clearly distinct converging subsequences 
\[\frac{\dim_\bbc \HH^0(\Gamma, \Sym^\lambda \bbc^2)}{\dim_\bbc \Sym^\lambda \bbc^2} = \begin{cases}
    1,&\text{if }\lambda\text{ is even}\,,\\
    0,&\text{if }\lambda\text{ is odd}\,,
\end{cases}\] where the parity of \(\lambda\) determines whether \(\Gamma\) acts on \(\Sym^\lambda \bbc^2\) through \(\bm{1}\) or \(\bm{-1}\).
\end{rmk}

We recall that a von Neumann algebra whose center is exactly \(\bbc\) is called a \emph{factor}. It is a well known fact that the center of \(\mathcal{N}(\Gamma)\) coincides with the closure in the weak topology of the center of \(\bbc[\Gamma]\). If \(Z \leq \Gamma\) is a central subgroup such that \(\Gamma/Z\) is an ICC group, i.e. every non-trivial conjugacy class is infinite, then \(Z(\bbc[\Gamma]) = \bbc[Z]\). Hence, if \(\Gamma/Z\) is ICC, the center of \(\mathcal{N}(\Gamma)\) is precisely \(\bbz[\Gamma]\).

\begin{lem}\label{lem-factors} Fix \(\chi \in \operatorname{Irr}(Z)\) and let \(\phi =\operatorname{ind}^\Gamma_Z \chi\). If \(\Gamma/Z\) is ICC then \(\mathcal{N}_\phi\) is a factor.
\end{lem}
\begin{proof} The algebra decomposition \[\mathcal{N}(\Gamma) \simeq \mathcal{N}(\Gamma) e_{\chi} \oplus \mathcal{N}(\Gamma)(1-e_{\chi})\] yields a decomposition \[Z(\mathcal{N}(\Gamma)) = \mathcal{Z}(\mathcal{N}(\Gamma))e_\chi \oplus Z(\mathcal{N}(\Gamma))(1-e_{\chi})\,.\] The identification of \(\mathcal{N}_\phi\) with the direct factor \(\mathcal{N}(\Gamma)e_\chi\) shows that the center \(Z(\mathcal{N}_\phi)\) is precisely the corner algebra \[Z(\mathcal{N}_\phi) =Z(\mathcal{N}(\Gamma))e_{\chi} = \bbc[Z]e_{\chi}\,,\] which is one-dimensional.
\end{proof}

\begin{rmk} The ``only if'' direction of the previous lemma is false in the twisted setting. For example, if \[\Gamma = Q_8 = \langle e,i,j,k\mid e^2 = 1\,,\, i^2 = j^2 = k^2 = ijk = e\rangle\] denotes the quaternion group, then its center \(Z = \langle e\rangle\) is cyclic of order two. Now, for the character \(\chi\colon Z \to \bbc^\times\) given by \(\chi(e) = -1\), one gets that \(\bbc[\Gamma]_{\chi}\) is a quaternion algebra over \(\bbc\) and, hence, a central simple \(\bbc\)-algebra. The finite groups \(\Gamma\) for which there exists a central character \(\chi\) such that \(\operatorname{Z}(\bbc[\Gamma]_\chi) = \bbc\) are called \emph{groups of central type}. 
\end{rmk}




We recall that a subset \(S\) of a ring \(R\) satisfies the \emph{right Ore condition} if it satisfies:
\begin{enumerate}
    \item for all \(r \in R\) and \(s \in S\), there exists \(r' \in R\) and \(s' \in S\) such that \(rs' = sr'\); and
    \item for all \(r \in R\) and \(s \in S\) such that \(sr = 0\), there exists \(s' \in S\) such that \(rs' = 0\).
\end{enumerate}
If \(S\) satisfies the right Ore condition, we can construct a ring \(RS^{-1}\) called the \emph{right Ore localization} of \(S\) with a flat ring homomorphism \(R \to RS^{-1}\) 
sending the elements of \(S\) to units and is universal among all ring homomorphisms that send \(S\) to units \cite[Lem 8.15]{luck_l2-invariants_2002}. We say that a ring \(R\) is \emph{von Neumann regular} if for any \(r \in R\) there exists \(s \in R\) such that \(srs  = r\). A \(*\)-algebra is called a \emph{\(*\)-regular ring} if it is von Neumann regular and its involution is proper, i.e. \(x^*x = 0\) if and only if \(x = 0\).

In the case of finite von Neumann algebras like \(\mathcal{N}_\phi\), it is proven in \cite[Prop. 2.8]{reich1998} that the set \(S\subseteq \mathcal{N}_\phi\) of all non zero-divisors satisfies the right Ore condition, and we can thus consider the localization $\mathcal{U}_\phi = \mathcal{N}_\phi S^{-1}$.

\begin{lem}\label{lem-affiliated-op} The following assertions are true:
\begin{enumerate}
    \item The map \(\iota\colon \mathcal{N}_\phi \to \mathcal{U}_\phi\) is injective.
    \item The involution \(*\), the von Neumann dimension \(\dim_\phi\) and the von Neumann rank \(\rk_\phi\) extend from \(\mathcal{N}_\phi\) to \(\mathcal{U}_\phi\).
    \item The ring \(\mathcal{U}_\phi\) is \(*\)-regular.
    \item If \(\mathcal{N}_\phi\) is a factor, then the center of \(\mathcal{U}_\phi\) is \(Z(\mathcal{N}_\phi) \simeq \bbc\), which equals the centralizer of the image of \(\Gamma\) in \(\mathcal{U}_\phi\).
\end{enumerate}
\end{lem}
\begin{proof} 
	Item (1) follows from the fact that the kernel of \(\iota\) consists of all \(r \in \mathcal{N}_\phi\) such that \(rs = 0\) for some \(s \in S\) and \(S\) consists of non-zero-divisors by the polar decomposition. The fact that the involution \(*\) extends to \(\mathcal{U}_\phi\) and turns it into a \(*\)-regular ring is shown in \cite[Prop. 2.10 and Note 2.11]{reich1998}. The extension of \(\dim_\phi\) and hence of \(\rk_\phi\) to \(\mathcal{U}_\phi\) follows from \cite[Prop. 3.9]{reich1998}. The last claim follows from \cite[Prop. 30]{liu_double_2012}.
\end{proof}

The ring \(\mathcal{U}_\phi\) can be identified with the ring of (unbounded) affiliated operators to \(\mathcal{N}_\phi\), from which we derive its name (see \cite[Sec. 2]{reich1998}). Observe that, if \(\Gamma'\) is a subgroup of \(\Gamma\) and \(\phi'\) is the restriction of \(\phi\) to \(\Gamma'\), then the inclusion map \(\mathcal{N}_{\phi'}\to \mathcal{N}_\phi\) induces an injective homomorphism \(\mathcal{U}_{\phi'} \to \mathcal{U}_\phi\) by the universal property of the Ore localization. 


\section{Crossed products and \texorpdfstring{\(*\)}{*}-regular closures}\label{sec-crossed-products}

We recall that a ring \(R'\) is called a \emph{crossed product} of a ring \(R\) and a group \(G\) if \(R\) can be identified with a subring of \(R'\) and there is an injective function \(\mu\colon G \to (R')^\times\) satisfying the following properties:
\begin{enumerate}
    \item (CrPr1) \(R'\) is a free left \(R\)-module on the set \(\mu(G)\);
    \item for all \(g \in G\), \(R\) is closed under conjugation by \(\mu(g)\); and
    \item (CrPr2) for every pair of elements \(g\) and \(g'\) in \(G\), the element \[\tau(g,g') = \mu(g)\mu(g')\mu(gg')^{-1}\] belongs to \(R^\times\).
\end{enumerate}

In particular, \(\mu\) induces a function \(\overline{\mu}\colon G \to \Aut(R)\) that satisfies the following:
\begin{enumerate}
    \item (CrPr1') for all \(g\), \(g'\) and \(g''\) in \(G\), we have that \[\tau(g,g')\tau(gg',g'') = \overline{\mu}(g)(\tau(g',g''))\tau(g,g'g'')\] in \(R^\times\); and
    \item (CrPr2') for all \(g\) and \(g'\) in \(G\) and \(r \in R\), we have that \[\tau(g,g')\overline{\mu}(gg')(r)\tau(g,g')^{-1} = \overline{\mu}(g)(\overline{\mu}(g')(r))\] in \(R\).
\end{enumerate}

Conversely, given \(\overline{\mu}\colon G \to \Aut(R)\) and \(\tau\colon G\times G \to R^\times\) satisfying (CrPr1') and (CrPR2'), we can construct a ring \(R*G\) with underlying set the free \(R\)-module with basis \(G\) and multiplication obtained by linearly extending \[(rg)\cdot (r'g') = \big(r\mu(g)(r')\tau(g,g')\big)gg'\,\] for $g,g'\in G$ and $r,r'\in R$. If \(R'\) is a crossed product of \(R\) and \(G\), then there is an isomorphism of rings \(R' \simeq R*G\) that restricts to the identity map on \(R\), where \(R*G\) is constructed using the induced maps from \(R'\). If \(\psi\colon R' \to \bigoplus_{g \in G} R\mu(g)\) is an \(R\)-module isomorphism, we say that \((\psi,\mu,\tau)\) is an \emph{(internal) twisting data}. A pair \((\overline{\mu},\tau)\) is an \emph{(external) twisting data}. The considerations above show that there exists a bijection between internal and external twisting data, and any twisting data determines a crossed product up to an \(R\)-ring isomorphism.

The crossed product satisfies a universal property: if \((\overline{\mu},\tau)\) is a twisting data, \(\varphi\colon R \to R'\) is any ring homomorphism and \(\nu\colon G \to (R')^\times\) is any function such that (1) \(\nu(g)\varphi(r)\nu(g)^{-1} = \varphi(\overline{\mu}(g)r)\) and (2) \(\varphi(\tau(g,g')) = \nu(g)\nu(g')\nu(gg')^{-1}\), then there exists a unique ring homomorphism \(\theta\colon R*G \to R'\) extending both \(\varphi\) and \(\nu\).

Suppose now \(R\) is contained in a larger ring \(R'\) and \((\overline{\mu},\tau)\) is a twisting data on \(R\). If there exists a function \(\overline{\mu'}\colon G \to \Aut(R')\) such that (a) for every \(g \in G\) and \(r \in R\) we have that \(\mu'(g)(r) = \mu(g)(r)\) and (b) the pair \((\overline{\mu'},\tau)\) satisfies (CrPr2') on \(R'\), then this pair also satisfies (CrPr1') and we get an injective map \(R*G \to R'*G\) extending the inclusion \(R \to R'\) and the identity \(G \to G\), where the crossed products are built using their respective data. In this case, we say that the crossed product structure on \(R\) \emph{extends} to a crossed product structure on \(R'\).

Keeping the notations of the previous section, we will now explore how the rings \(\mathcal{N}_\phi\) and \(\mathcal{U}_\phi\) decompose as crossed products over finite quotients of \(\Gamma\). Let \(\Gamma'\) be a finite index normal subgroup of \(\Gamma\) and, if \(\chi\colon Z = Z(\Gamma) \to \bbc^\times\) is a central character, we denote by \(Z' = Z\cap \Gamma'\) and \(\chi'\colon Z' \to \bbc^\times\) the restriction of this character to \(\Gamma'\). We have \(\phi = \operatorname{ind}^{\Gamma}_Z\chi\) and \[\phi' = \operatorname{ind}^{\Gamma'}_{\chi'} \chi' = \phi_{\mid \Gamma'}\,.\] By taking \(\Gamma_\phi = \Gamma/Z\) and \(\Gamma'_{\phi'} = \Gamma'/Z'\), observe that we have an embedding \(\Gamma'_{\phi'} \unlhd \Gamma_\phi\) induced by the inclusion, and the quotient \(\Gamma_\phi/\Gamma'_{\phi'}\) is isomorphic to \(\Gamma/(\Gamma'Z)\). 

Recall the definition of the central idempotents of \(\bbc[\Gamma]\) and \(\bbc[\Gamma']\) associated to \(\chi\): \[e_{\chi} = \frac{1}{|Z|}\sum_{h \in Z} \chi(h^{-1})h \in \bbc[Z]\,,\quad e_{\chi'} = \frac{1}{|Z'|}\sum_{h \in Z'} \chi(h^{-1})h \in \bbc[Z']\,.\] The image of \(\bbc[\Gamma]\) (resp. \(\bbc[\Gamma'])\) inside \(\mathcal{N}_\phi\) (resp. \(\mathcal{N}_{\phi'}\)) is precisely the corner algebra \(\bbc[\Gamma']e_\chi\) (resp. \(\bbc[\Gamma]e_{\chi'}\)).

\begin{lem}\label{lem-corner-algebra-iso} The subset \(\bbc[\Gamma']e_{\chi}\) is a subalgebra of \(\bbc[\Gamma]e_{\chi}\), and the algebra map \[\bbc[\Gamma']e_{\chi'} \to \mathbb{C}[\Gamma']e_\chi\] given by right multiplication by \(e_{\chi}\) is an isomorphism.
\end{lem}
\begin{proof} The fact that \(\bbc[\Gamma']e_{\chi}\) is closed under products follows immediately from \(e_{\chi}\) being a central idempotent of \(\bbc[\Gamma]\). Since any \(h \in Z\) acts on \(e_{\chi}\) as multiplication by \(\chi(h)\), one sees that \[e_{\chi'}e_{\chi} = \left(\frac{1}{|Z'|}\sum_{h \in Z'} \chi(h^{-1})\chi(h)\right)e_\chi = e_\chi\,.\] Hence, right multiplication by \(e_{\chi}\) is an algebra homomorphism from \(\bbc[\Gamma']e_{\chi'}\) to \(\bbc[\Gamma']e_{\chi}\). It is surjective by definition (being the image of \(\bbc[\Gamma']\) under right multiplication by \(e_\chi\)).

To see that it is injective, choose a transversal \(T\) of \(Z/Z'\) and write \[e_{\chi} = \frac{1}{[Z\colon Z']}\sum_{h \in T} \chi(h^{-1})he_{\chi'}\,.\] For \(x \in \bbc[\Gamma']\) in the kernel of multiplication by \(e_\chi\), we have: \[0 = xe_\chi = \frac{1}{[Z\colon Z']} \sum_{h \in T} \chi(h^{-1}) (hxe_{\chi'})\,.\] For the distinct elements \(h \in T\), the elements \(hxe_{\chi'}\) are supported on distinct cosets of \(\Gamma'\) and therefore must all vanish. This implies that \(xe_{\chi'} = 0\), as desired.
\end{proof}

\begin{lem}\label{lem-first-crossed-prod} The following are true:
\begin{enumerate}
    \item The conjugation action of \(\Gamma\) on \(\Gamma'\) induces a twisting data \(\overline{\mu}\colon \Gamma_\phi/\Gamma'_{\phi'} \to \Aut(\bbc[\Gamma']e_{\chi'})\) and \(\tau\colon (\Gamma_\phi/\Gamma'_{\phi'})^2 \to (\bbc[\Gamma']e_{\chi'})^\times\), and \(\bbc[\Gamma]e_\chi \simeq \bbc[\Gamma']e_{\chi'} * (\Gamma_\phi/\Gamma'_{\phi'})\) under the right multiplication by \(e_\chi\) map.
    \item The crossed product structure on \(\bbc[\Gamma']e_{\chi'}\) extends to a crossed product structure on \(\mathcal{N}_{\phi'}\) and \(\mathcal{U}_{\phi'}\).
    \item The subring of \(\mathcal{U}_\phi\) generated by \(\mathcal{U}_{\phi'}\) and \(\Gamma\) is isomorphic to the crossed product \(\mathcal{U}_{\phi'} * (\Gamma_{\phi}/\Gamma'_{\phi'})\).
\end{enumerate}
\end{lem}
\begin{proof} (1) Fix a set of representatives \(T'\) for \(\Gamma'/Z'\) and \(X\) for \(\Gamma/(\Gamma'Z)\), so that \(T = T'X\) is a transversal for \(\Gamma/Z\) and the projection \(\Gamma \to \Gamma_{\phi}/\Gamma'_{\phi'}\) is bijective when restricted to \(X\). For each \(g \in \Gamma_{\phi}/\Gamma_{\phi'}\), let \(x_g\) be its representative in \(X\). For any pair \(g\) and \(g'\) in \(\Gamma_{\phi}/\Gamma'_{\phi'}\), there is a unique \(u = u(g,g') \in \Gamma'Z\) such that \(x_gx_{g'} = ux_{gg'}\). Since \(\Gamma'/Z'\) is isomorphic to \((\Gamma'Z)/Z\) one can also find unique elements \(h = h(g,g') \in T'\) and \(z = z(g,g') \in Z\) such that \(u = zh\).

Define \(\overline{\mu} \in \Aut(\bbc[\Gamma']e_{\chi'})\) induced by the conjugation by \(x_g\), which is well defined since \(1-e_{\chi'}\) is central in \(\bbc[\Gamma]\). Set \(\tau(g,g') = \chi(z)he_{\chi'}\) where \(z\) and \(h\) are defined as above. A direct computation shows that \((\overline{\mu},\tau)\) satisfies (CrPr1') and (CrPr2') and thus we can construct the crossed product \(\bbc[\Gamma']e_{\chi'}*(\Gamma_\phi/\Gamma'_{\phi'})\). The right-multiplication by \(e_{\chi}\) map \(\bbc[\Gamma']e_{\chi'} \to \bbc[\Gamma]e_{\chi}\) and the function \(\nu\colon \Gamma_\phi/\Gamma'_{\phi'} \to (\bbc[\Gamma]e_{\chi})^\times\) that sends \(g\) to \(x_ge_\chi\) induce, by the universal property of the crossed product, a surjective ring homomorphism \(\bbc[\Gamma']e_{\chi'}*(\Gamma_\phi/\Gamma'_{\phi'}) \to \bbc[\Gamma]e_\chi\). To see that it is injective, we must check that the elements \(\{x_g \colon g \in \Gamma_{\phi}/\Gamma'_{\phi'}\}\) are linearly independent over \(\bbc[\Gamma']e_\chi\). This is immediate from the action of \(\bbc[\Gamma]e_\chi\) on \[\mathcal{H}_\phi = \bigoplus_{g \in \Gamma_\phi/\Gamma_{\phi'}} x_g\mathcal{H}_{\phi'}\,,\] in which the \(x_g\) freely permute the direct summands up to an automorphism of each and the elements of \(\bbc[\Gamma']e_{\chi}\) stabilize each summand.  

(2) Since the image of \(x_g\) in both \(\mathcal{N}_\phi\) and \(\mathcal{U}_\phi\) consists of units that normalize \(\Gamma'\), one only needs to show that \((\overline{\mu},\tau)\) also satisfies (CrPr2') on \(\mathcal{N}_{\phi'}\) and \(\mathcal{U}_{\phi'}\) respectively. For \(\mathcal{N}_{\phi'}\) it suffices to show that the actions of both sides of (CrPr2'), seen as operators of \(\mathcal{H}_{\phi}\), coincide on the dense subspace \(\bbc[\Gamma']e_{\chi'}\). In fact, it suffices to show that they send the identity element to the same element, which follows from any \(r \in \mathcal{N}_{\phi'}\) acting on \(\bbc[\Gamma']e_{\chi'}\) as multiplication by some element in \(\ell^2(\Gamma')\). For \(\mathcal{U}_{\phi'}\), observe that both sides of (CrPr2') define automorphisms of \(\mathcal{U}_{\phi'}\) that coincide on \(\mathcal{N}_{\phi'}\), so they must be the same.

(3) This subring coincides with the image of the induced map \(\varphi\colon \mathcal{U}_{\phi'}* (\Gamma_\phi/\Gamma'_{\phi'}) \to \mathcal{U}_{\phi}\), so it suffices to prove that this map is injective, that is, that the set \(X\) is linearly independent over \(\mathcal{N}_{\phi'}\) as operators on \(\mathcal{H}_{\phi}\). This follows as in (1) from the identification of \(\mathcal{H}_{\phi}\) with the Hilbert space \(\bigoplus_{g \in \Gamma_\phi/\Gamma'_{\phi'}} x_g\mathcal{H}_{\phi'}\).
\end{proof}

We recall that in a \(*\)-regular ring \(R'\), for every element \(r \in R'\) there are unique projections \(e\) and \(f\) such that \(R'e = R'r\) and \(fR' = rR'\) \cite[Prop. 3.2(3)]{jaikin-zapirain_base_2019}. Moreover, there exists a unique element \(r^{[-1]}\), called the relative inverse of \(r\), such that \(r^{[-1]}r = e\) and \(rr^{[-1]} = f\) \cite[Prop. 3.2(4)]{jaikin-zapirain_base_2019}.

If \(R\) is a subring of a \(*\)-regular ring \(R'\), then there exists a unique minimal \(*\)-regular subring \(\mathcal{R}(R,R')\) of \(R'\) containing \(R\), called the \emph{\(*\)-regular closure} of \(R\) in \(R'\) \cite[Sec. 3.4]{jaikin-zapirain_base_2019}. This ring is obtained inductively by setting \(\mathcal{R}_0(R,R') = R\) and taking \(\mathcal{R}_{i+1}(R,R')\) to be the subring generated by \(\mathcal{R}_i(R,R')\) and the relative inverses of all the elements \(rr^*\) for \(r \in \mathcal{R}_i(R,R')\), see \cite[Prop. 6.2]{ara_realization_2017}.

\begin{lem} Let \(L\) be a subfield of \(\bbc\) containing \(L_0\) and \(\Gamma'\) be a finite index normal subgroup of \(\Gamma\). Then, we also get a decomposition \(L[\Gamma]e_{\chi} \simeq L[\Gamma']e_{\chi'}*(\Gamma_\phi/\Gamma'_{\phi'})\), the crossed product structure on \(L[\Gamma']e_{\chi'}\) extends to \[\mathcal{R}(\Gamma')_{\phi'}^L = \mathcal{R}(L[\Gamma']e_{\chi'},\mathcal{U}_{\phi'})\] and the subring of \(\mathcal{R}(\Gamma)_\phi^L\) generated by \(\mathcal{R}(\Gamma')_{\phi'}^L\) and \(\Gamma\) is isomorphic to \(\mathcal{R}(\Gamma')_{\phi'}^L * (\Gamma_\phi/\Gamma'_{\phi'})\).
\end{lem}
\begin{proof} The decomposition for \(L[\Gamma]e_{\chi}\) is proven in the same way as for \(\bbc[\Gamma]e_{\chi}\) in Lemma~\ref{lem-first-crossed-prod}(1).

Assume now that we have already extended the twisting data for \(L[\Gamma']_{\chi'}\) to \(\mathcal{R}_i(L[\Gamma']e_{\chi'},\mathcal{U}_{\phi'})\). To show that it extends to \(\mathcal{R}_{i+1}(L[\Gamma']e_{\chi'},\mathcal{U}_{\phi'})\), all we must show is that this ring is invariant under conjugation by \(x \in X\), where \(X\) is the chosen transversal of \(\Gamma_\phi/\Gamma'_{\phi'}\) for the twisting data, since it will inherit (CrPr2') from the crossed product decomposition of \(\mathcal{U}_{\phi}\). 

Hence, let \(r \in \mathcal{R}_{i}(L[\Gamma']e_{\chi'},\mathcal{U}_{\phi'})\) be any self-adjoint element and \(e,f \in \mathcal{U}_{\phi'}\) be projections such that \(\mathcal{U}_{\phi'}r = \mathcal{U}_{\phi'}e\) and \(r\mathcal{U}_{\phi'} = f\mathcal{U}_{\phi'}\). Then, for any \(x \in X\), we have that \(xex^{-1}\) and \(xfx^{-1}\) must be the unique projections such that \[\mathcal{U}_{\phi'}(xrx^{-1}) = x(\mathcal{U}_{\phi'}r)x^{-1} = x(\mathcal{U}_{\phi'}e)x^{-1} = \mathcal{U}_{\phi'}(xex^{-1})\,,\]
\[(xrx^{-1})\mathcal{U}_{\phi'} = x(r\mathcal{U}_{\phi'})x^{-1} = x(f\mathcal{U}_{\phi'})x^{-1} = (xfx^{-1})\mathcal{U}_{\phi'}\,.\]
This shows that \(xr^{[-1]}x^{-1}\) must be the relative inverse of \(xrx^{-1}\). Hence, we have that \[xr^{[-1]}x^{-1} = (xrx^{-1})^{[-1]} \in \mathcal{R}_{i+1}(L[\Gamma']e_{\chi'},\mathcal{U}_{\phi'})\] as desired.

From the universal property of crossed products, we get a ring homomorphism \(\theta\colon \mathcal{R}(\Gamma')_{\phi'}^L * (\Gamma_{\phi}/\Gamma'_{\phi'}) \to \mathcal{U}_{\phi}\) which is the restriction of the injective map \(\mathcal{U}_{\phi'} * (\Gamma_\phi/\Gamma'_{\phi'}) \to \mathcal{U}_{\phi}\), proving the last claim.
\end{proof}

\begin{rmk}\label{rmk-center-*-regular} Observe that if \(\mathcal{N}(\Gamma)_\chi\) is a factor, then by Lemma~\ref{lem-affiliated-op}(e) the center of \(\mathcal{R}(\Gamma)_{\chi}^L\) is contained in the center of \(\mathcal{U}(\Gamma)_{\chi}\), and is therefore a subfield of \(\bbc\).
\end{rmk} 

We are now able to describe how we are going to make use of Lemma~\ref{lem-factors}. Conjecture~\ref{lie-luck-conj} states that for any finitely generated subgroup \(\Gamma\) of \(\schG\), for every sequence \(\{\scheme W_k\}\) of representations of \(\schG\) with the same central character \(\chi\colon \scheme Z (\schG) \cap \Gamma \to \bbc^\times\) such that their highest weight parameters \(\bm\lambda = (\lambda_1,\ldots,\lambda_n)\) (\(\lambda_i = \lambda_i(k)\)) grow to infinity, and for every finitely generated subfield \(L\) of \(\bbc\) containing \(L_0\), we have that \[\lim_{k\to \infty} \rk_{\Wmodk{L}}^{\Gamma} = \rk_{\Gamma}^\chi\] as rank functions. Indeed, for each matrix \(A\) over \(\bbc[\Gamma]\), it suffices to take \(L\) the subfield generated by the coefficients of those entries, as we have that \(\rk^\schG_{\Wmodk{\bbc}} A = \rk^\Gamma_{\Wmodk{L}} A\). We may further assume that \(\Gamma\) is Zariski-dense, so that \(Z\) coincides with the center of \(\Gamma\) and \(\Gamma/Z\) is ICC.

Suppose now that we can find a normal subgroup \(\Gamma'\) of \(\Gamma\) of finite index such that, if \(\chi'\) denotes the restriction of \(\chi\) to \(Z\cap \Gamma'\), we have the convergence \[\lim_{k \to \infty} \rk^{\Gamma'}_{\Wmodk{L}} = \rk_{\phi'}\] of Sylvester matrix rank functions on \(L[\Gamma']\). We will show that this already implies the desired convergence for \(\Gamma\) over \(L\).

\begin{thm}\label{thm-first-red} Let \(\Gamma\), \(\Gamma'\), \(\chi\), \(\chi'\), \(\phi\), \(\phi'\) and \(\Wmodk{}\) be as above. If \(\rk^{\Gamma'}_{\Wmodk{L}}\) converges to \(\rk_{\phi'}\), then \(\rk^{\Gamma}_{\Wmodk{L}}\) converges to \(\rk_{\phi}\).
\end{thm}
\begin{proof} Let \(\omega\) be a non-principal ultrafilter on \(\bbn\), and consider the ultraproduct of tracial \(*\)-regular rings \[\mathcal{D} = \prod_{\omega} \End_{\bbc}(\Wmodk{})\] as described in \cite[Sec. 6.7]{jaikin-zapirain_base_2019}. The representations of \(\Gamma'\) induce a homomorphism \(L[\Gamma'] \to \mathcal{D}\) whose \(*\)-regular closure we denote by \(\mathcal{S}\). The fact that \[\rk_{\phi'} = \lim_{\omega} \rk^{\Gamma'}_{\Wmodk{L}}\] as rank functions over \(L[\Gamma']\) implies by \cite[Thm. 6.3]{jaikin-zapirain_base_2019} that we have a rank preserving isomorphism \(\varphi\colon \mathcal{R}(\Gamma')^L_{\phi'} \to \mathcal{S}\) that commutes with the maps \(L[\Gamma'] \to \mathcal{R}(\Gamma')^L_{\phi'}\) and \(L[\Gamma'] \to \mathcal{S}\).

The representations of \(\Gamma\) also induce a homomorphism \(L[\Gamma] \to \mathcal{D}\) that clearly extends the map \(L[\Gamma'] \to \mathcal{D}\). By the universal property of the crossed product, we obtain a ring homomorphism \(\theta\colon \mathcal{R}(\Gamma)_{\phi}^L \simeq \mathcal{R}(\Gamma')_{\phi'}^L * (\Gamma_\phi/\Gamma'_{\phi'}) \to \mathcal{D}\) that extends \(\varphi\). Since \(\Gamma\) is Zariski-dense, \(\Gamma/Z\) is ICC and thus, by Lemma~\ref{lem-factors} and Remark~\ref{rmk-center-*-regular}, the center of \(\mathcal{R}(\Gamma)_{\phi}^L\) is a subfield of \(\bbc\). 

It was shown in \cite[Thm. 19.13, 19.14]{goodearl_von_1979} (cf. \cite[Prop. 5.7]{jaikin-zapirain_base_2019}) that \(\mathcal{R}(\Gamma)_{\phi}^L\) admits at most one Sylvester matrix rank function. Since both \(\rk_{\phi}\) and \(\lim_\omega \rk^\Gamma_{\Wmodk{L}}\) define Sylvester matrix rank functions on \(\mathcal{R}(\Gamma)_{\phi}^L\), they must also coincide. The result for general limits then follows, as the choice of \(\omega\) was arbitrary.
\end{proof}

\begin{cor}\label{cor-simplified-conjecture} Conjecture~\ref{lie-luck-conj} holds for any finitely generated subgroup of \(\schG\) if and only if for every finitely generated and Zariski-dense subgroup \(\Gamma\) of \(\schG\) with \(\Gamma \leq \schG(L)\) and \(L \leq \bbc\) finitely generated, there exists a finite index normal subgroup \(\Gamma'\) of \(\Gamma\) such that \[\lim_{k \to \infty} \rk^{\Gamma'}_{\scheme W_k} = \rk_{\phi'}\] whenever all the central characters of \(\scheme W_k\) are equal to \(\chi\), the parameters of the highest weight \(\lambda_i = \lambda_i(k)\) grow to infinity and \(\chi'\) is the restriction of \(\chi\) to \(\scheme{Z}(\schG) \cap \Gamma'\).
\end{cor}

Finally, we end with a proposition that allows us to lift information about rank functions on \(\Gamma/Z'\) to information about rank functions on \(\Gamma\) when \(Z' \leq \ker\chi\). 

\begin{prop}\label{prop:lifting-matrices} Let \(\Gamma\), \(Z\) and \(\chi\) be as above, choose \(Z' \leq \ker \chi\) and define \(\Gamma' = \Gamma/Z'\) with the induced character \(\chi'\colon Z/Z' \to \mathbb{C}^\times\). Denote by \(\phi = \operatorname{ind}^\Gamma_Z \chi\) and \(\phi' = \operatorname{ind}^{\Gamma'}_{Z/Z'} \chi'\) the induced characters. Let \(s\colon \Gamma' \to \Gamma\) be a section such that \(s(1) = 1\). Then, for any matrix \(A\) over \(\bbc[\Gamma']\), if we denote by \(s(A)\) the induced matrix over \(\bbc[\Gamma]\) given by applying the section \(s\) to \(A\) entrywise, we have \[\rk_\phi(s(A)) = \rk_{\phi'}(A)\,.\]
\end{prop}
\begin{proof} Let \(\pi\colon \bbc[\Gamma] \to \bbc[\Gamma']\) be the projection map and \(s\colon \bbc[\Gamma'] \to \bbc[\Gamma]\) be the injective linear extension of \(s\) to a map of \(\bbc\)-vector spaces. We now claim that \(\pi\) induces an isomorphism \(\bbc[\Gamma]e_{\chi} \to \bbc[\Gamma']e_{\chi'}\). This map is clearly surjective, and if \(x \in \bbc[\Gamma]\) is such that \(\pi(xe_\chi) = 0\) vanishes, then \(xe_\chi\) can be written as a \(\bbc[\Gamma]\)-linear combination of the elements \(h-1\) as \(h\) runs through \(Z'\): \[xe_\chi = \sum_{h \in Z'} x_h(h-1)\,.\] However, multiplying by \(e_\chi\), on both sides and using that \(he_\chi = e_\chi\), we get \[xe_\chi = (xe_\chi)e_\chi = \sum_{h \in Z'}x_h(h-1)e_\chi = 0\,.\] It's inverse is given by: \[\pi^{-1}(xe_{\chi'}) = s(x)e_{\chi}\,,\] which can be checked by applying \(\pi\) to both sides -- the ambiguity in the choice of \(s\) vanishes after we multiply by \(e_{\chi}\).

If we let \(\operatorname{tr}_\phi\) and \(\operatorname{tr}_{\phi'}\) be the respective traces on \(\bbc[\Gamma]e_\chi\) and \(\bbc[\Gamma']e_{\chi'}\), we  now claim that for any \(xe_{\chi} \in \bbc[\Gamma]e_{\chi}\), we have \[\operatorname{tr}_{\phi}(xe_{\chi}) = \operatorname{tr}_{\phi'}(\pi(x)e_{\chi'})\,.\] Indeed, since \(\operatorname{tr}\) is a linear map, it suffices to check this equality for \(x \in \Gamma\). If \(x \not\in Z\), then \(\pi(x) \not\in Z/Z'\) and both sides of the equation are zero. If \(x \in Z\), then since \(\chi = \chi' \circ \pi\) we have \[\operatorname{tr}_{\phi}(xe_\chi) = \chi(x) = \chi'(\pi(x)) = \operatorname{tr}_{\phi'}(\pi(x)e_{\chi'})\,.\]

Hence, \(\pi\colon \bbc[\Gamma]e_\chi \to \bbc[\Gamma']e_{\chi'}\) is an isomorphism of algebras that preserves the trace. In particular, for any matrix \(Ae_{\chi}\) over \(\bbc[\Gamma]e_{\chi}\) whose projection matrix onto its image is \(Pe_{\chi}\), the matrix \(\pi(P)e_{\chi'}\) is the projection onto the image of the matrix \(\pi(A)e_{\chi'}\) over \(\bbc[\Gamma']e_{\chi'}\). Therefore, we obtain \[\rk_{\phi}(Ae_{\chi}) = \operatorname{tr}_{\phi}(Pe_\chi) = \operatorname{tr}_{\phi'}(\pi(P)e_{\chi'})  = \rk_{\phi'}(\pi(A)e_{\chi'})\,.\] Now, for any matrix \(A\) over \(\bbc[\Gamma']\), we get:
\begin{align*}
    \rk_{\phi'}(A) &= \rk_{\phi'}(Ae_{\chi'})\\
    &= \rk_{\phi'}(\pi(\pi^{-1}(Ae_{\chi'})))\\
    &= \rk_{\phi}(\pi^{-1}(Ae_{\chi'}))\\
    &= \rk_{\phi}(s(A)e_{\chi}) = \rk_{\phi}(s(A))\,.\qedhere
\end{align*}

\end{proof}

\bibliographystyle{alpha}
\bibliography{bib}

\end{document}